\documentclass[a4paper,11pt]{amsart}
\usepackage[margin=0.94in]{geometry}
\usepackage[utf8]{inputenc}
\usepackage{amsfonts,amssymb,amsmath,amstext,amscd,epsfig}
\usepackage{amsthm,cases}
\usepackage{graphicx}
\usepackage{stackengine}
\usepackage[colorlinks=true,pdfborder={0 0 0}]{hyperref}
\hypersetup{urlcolor=blue,citecolor=teal,linkcolor=purple}
\numberwithin{equation}{section}
\newtheorem{lemma}{Lemma}[section]
\newtheorem{theorem}[lemma]{Theorem}

\newtheorem{proposition}[lemma]{Proposition}
\newtheorem{remark}{Remark}
\usepackage{tikz}

\stackMath
\newcommand{\R}{\mathbb{R}}

\newcommand{\Z}{\mathbb{Z}}
\newcommand{\T}{\mathbb{T}}
\newcommand{\G}{\mathbb{G}}

\newcommand{\N}{\mathbb{N}}
\newcommand{\sX}{\mathsf{X}}

\newcommand{\dd}{\, {\rm d}}
\newcommand{\dt}{\frac{{\rm d}}{{\rm d}t}}
\newcommand{\ds}{\displaystyle}
\newcommand{\cC}{\mathcal{C}}
\newcommand{\cR}{\mathcal{R}}
\newcommand{\cL}{\mathcal{L}}
\newcommand{\cE}{\mathcal{E}}
\newcommand{\cM}{\mathcal{M}}
\newcommand{\cO}{\mathcal{O}}
\newcommand{\fB}{\mathfrak{B}}
\newcommand{\Lip}{\operatorname{Lip}}
\newcommand{\var}{\varepsilon}
\renewcommand{\Re}{\operatorname{Re}}

\title[A consistency-stability approach to scaling limits\dots]{A
  consistency-stability approach to scaling limits of zero-range processes}

\author{Daniel Marahrens}
\address{WorldQuant, London, UK}
\email{daniel.marahrens@gmail.com}

\author{Angeliki Menegaki}
\address{Department of Mathematics, Imperial College London, UK}
\email{a.menegaki@imperial.ac.uk}

\author{Clément Mouhot}
 \address{DPMMS, University of Cambridge, UK}
\email{c.mouhot@dpmms.cam.ac.uk}

\begin{document}
\maketitle

\begin{abstract}
  We propose a simple quantitative method for studying the hydrodynamic limit of interacting particle systems on lattices. It is applied to the diffusive scaling of the symmetric Zero-Range Process (in dimensions one and two). The rate of convergence is estimated in a Monge-Kantorovich distance asymptotic to the $L^1$ stability estimate of Krüzkhov, as well as in relative entropy; and it is uniform in time. The method avoids the use of the so-called ``block estimates''. It is based on a modulated Monge-Kantorovich distance estimate and microscopic stability properties.
\end{abstract}

\tableofcontents

\section{Introduction}

We consider the hydrodynamic limit of interacting particle systems on a lattice. The general problem is to show that under an appropriate scaling of time and space, the local particle density of a stochastic lattice gas converges to the solution to a macroscopic partial differential equation. We present the method at an abstract level and then we apply it to a prototypical concrete model, the so-called Zero-Range Process (ZRP). 

\subsection{State of the art}

Hydrodynamic limits are known at a qualitative level under both hyperbolic and parabolic scalings for several stochastic models, including the simple-exclusion processes, exclusion processes with degenerate rates, the zero-range processes, as well as spin models such as the Ginzburg-Landau process with convex potential or with Kawasaki dynamics, see~\cite{GPV88,Yau91,Reza,KL99,PatriciaLandim09} and references threrein. These proofs rely on entropy estimates, either with respect to the global invariant mesure as in \textit{entropy methods}, or with respect to the local Gibbs measure as in \textit{relative entropy methods} introduced for such models by~\cite{Yau91}. However finding quantitative error estimates has been an area of active research for several decades, and has remained opened in most cases. Such quantitative error estimates are also connected to the related issue of the long-time behaviour of this hydrodynamic limit.

Let us discuss the existing results of quantitative hydrodynamic limit. For the $d$-dimensional simple exclusion process with weak asymmetry, for all $d\geq 1$, the optimal rate was obtained in~\cite{JM18}, together with the fluctuation limit, by refining the relative entropy method and without relying on logarithmic Sobolev inequalities. This approach has not been applied to truly nonlinear process, i.e. with nonlinear jump rates, such as the zero-range process, to our knowledge. Quantitative fluctuations were recently obtained in~\cite{BenVitalii24} for the symmetric simple exclusion process, with the optimal convergence rate. We also note the recent results on quantitative hydrodynamical limits via homogenisation techniques for non-gradient interacting particle systems, first in \cite{Giunti:2022aa} and more recently in \cite{FGW24}. Regarding spin processes, the article~\cite{YauH-1} considered the $1$-dimensional convex Ginzburg-Landau process and obtained the optimal rate, via a modulated $H^{-1}$ estimate. In the non-convex case, the only quantitative result so far seems to be the nice recent paper~\cite{DMOW18a} which obtains a quantitative polynomial rate in dimension one by building upon previous ideas of~\cite{GOVW09}: it relies on establishing a uniform logarithmic Sobolev inequalities and the introduction of an intermediate mesoscopic; it seems to be able to capture the optimal convergence rate. 

The relative entropy method in~\cite{Yau91} can be made quantitative, provided that the \textit{block estimates} it relies on are also made quantitative: it is feasible and interesting in its own right, and we will present it in a separate forthcoming work~\cite{MMM:sequel}. Note that all these works assume and rely on the regularity of the limit solution. This excludes lattice models who limit solution present interesting degenerate solutions, such as shocks, for conservation laws, or compactly supported profiles, for porous medium equation. Our goal is to apply the new method presented here to the question of dealing with such degenerate limit solutions in the future.

In this article, we consider reversible processes and the diffusive (parabolic) scaling. We propose a systematic short method for proving the hydrodynamic limit for symmetric and attractive processes  with quantitative rate, in weak distance and in relative entropy. Our method relies on a consistency-stability approach, based on, roughly, three key properties: macroscopic and microscopic stability estimates in appropriate topologies, and consistency estimates comparing the microscopic and the macroscopic semigroups. This approach thus differs in spirit with the typical techniques and remarkably, it avoids the use of the block estimates (see~\cite[Chapter~5]{KL99}). We establish microscopic stability estimates reminiscent of Kruzhkov-type stability estimates on the limit equation and inspired from~\cite{Reza}. We also derive higher-regularity microscopic stability estimates, referred to as "jump distance estimates", during the proof of the consistency error. We take advantage of the relaxation of the macroscopic dynamics to prove that our estimates are uniform in time, under parabolic scaling. 
We use a Poincar\'{e} inequality in the consistency estimate, and we do not require the validity of a logarithmic Sobolev inequality. 
To our knowledge, this is the first quantitative result for the Zero-Range Process when it is truly nonlinear, i.e. with nonlinear jump rate. The hydrodynamic limit for degenerate nonlinear diffusion equation (see \cite{PatriciaLandim09} and references therein) is currently under investigation where our method seems to yield first promising quantitative results.

\subsection{Acknowledgements} The authors thank Thierry Bodineau, Max Fathi, Pierre Le Bris, Jean-Christope Mourrat, Stefano Olla, Felix Otto and Jeremy Quastel for enlightening discussions or for suggesting references. The second author acknowledges partial support from a Huawei fellowship at Institut des Hautes Études Scientifiques (IHES) and a Chapman fellowship at Imperial College London. The third author acknowledges partial funding from the ERC (under project MAFRAN). Both the second and third authors gratefully acknowledge the support of IHES where part of this work was completed.

\section{The general method}
\label{sec:intro}

\subsection{Set up and notation}

We denote by $\sX \subset \R$ the state space at a given site
(number of particles, spin, etc.). 
Consider the discrete
torus $\T_N^d := \{ 0,1, \dots, N-1 \}^d$ and the corresponding
phase space of particle configurations $\sX_N:=\sX^{\T^d
  _N}$. Variables $x,y,z \in \T^d_N$ are called
\emph{microscopic}, whereas variables in the continuous unit
torus $u, v, w \in \T^d$ are called \emph{macroscopic}; finally
particle \emph{configurations} are denoted by
$\eta, \zeta \in \sX_N$. We embed $\T_N^d$ into $\T^d$ by the map
$x \mapsto N^{-1}x$, which means that the distance between sites
of the lattice scales like $N^{-1}$. Given a particle
configuration $\eta = (\eta_x)_{x \in \T^d_N} \in \sX_N$, we
define the \emph{empirical measure}
\begin{equation}
  \label{eq:emp_meas_dirac}
  \alpha_\eta^N := N^{-d} \sum_{x\in\T^d_N} \eta_x
  \delta_{\frac{x}{N}} \in \mathcal{M}_+(\T^d)
\end{equation}
where $\mathcal{M}_+(\T^d)$ is the space of positive Radon
measures on the torus.

At the microscopic level, the interacting particle system evolves
through a stochastic process and the time-dependent probability
measure describing the law of $\eta$ is denoted by
$\mu_t^N \in P(\sX_N)$, where $P(\sX_N)$ denotes the space of
probability measures on $\sX_N$. We consider a \emph{linear}
operator $\cL_N : C_b(\sX_N) \rightarrow C_b(\sX_N)$ generating a
unique Feller semigroup $e^{t \cL_N}$ on $P(\sX_N)$
(see~\cite[Chapter 1]{Liggett1985} for the definitions) so that
given $\mu_0^N \in P(\sX_N)$ the law
$\mu_t ^N = e^{t \cL_N ^\dag} \mu^N_0 \in P(\sX_N)$, with
$\cL_N ^\dag$ the adjoint acting on $P(\sX_N)$, satisfies
\begin{equation}
  \label{eq:particle_evol}
  \forall \, \Phi \in C_b(\sX_N), \quad
  \dt\langle \Phi, \mu^N_t \rangle = \langle \cL_N \Phi, \mu^N_t
  \rangle,
\end{equation}
where $C_b(\sX_N)$ denotes continuous bounded real-valued
functions over $\sX_N$, and $\langle \cdot, \cdot \rangle$
denotes the duality bracket between $C_b(\sX_N)$ and
$P(\sX_N)$. 

At the macroscopic level, we consider a map (possibly depending
on $N$)
\begin{equation*}
  \mathbf{L}_\infty  : \cM_+(\T^d) \rightarrow \cM_+(\T^d),
\end{equation*}
and in general \emph{unbounded} and \emph{nonlinear}, and the
evolution problem on $f(t,u)=f_t(u)$:
\begin{equation}
  \label{eq:limitPDE1} 
  \partial_t f_t  = \mathbf{L}_\infty   f_t , \quad f_{t =0}=f_0.
\end{equation}

A measure $\mu^N_\infty \in P(\sX_N)$ is called \emph{invariant}
for~\eqref{eq:particle_evol} if
\begin{equation*}
  \forall \, \Phi \in C_b(\sX_N), \quad
  \big\langle \mathcal{L}_N \Phi, \mu^N _\infty \big\rangle = 0.
\end{equation*}
We consider a class of models where the equilibrium has the
following \emph{product structure}. For any
$\lambda \in \text{Conv}(\sX)$, where
$\text{Conv}(\sX) \subset \R$ is the convex hull of $\sX$, there
is a one-site marginal probability measure 
\begin{equation*}
    n_\lambda: \text{Conv}(\sX) \to \R_+ \quad \text{ and }
    \quad \sigma : \text{Conv}(\sX) \to \text{Conv}(\sX)
\end{equation*}
so that for all $\lambda \in \text{Conv}(\sX)$
\begin{equation}
  \label{eq:inv-meas-gen}
  \mu^N_{\infty,\lambda} := n_\lambda ^{\otimes \T^d_N}, \quad
  \text{i.e.} \quad \mu^N_{\infty,\lambda}(\eta) = \prod_{x \in
    \T^d_N} n_\lambda(\eta_x),
\end{equation}
is an invariant measure on $\sX_N$, and for any
$\rho \in \text{Conv}(\sX)$,
$\mathbb E_{n_{\sigma(\rho)}}[\eta_x] = \rho$. We then define,
given a macroscopic profile $f : \T^d \to \text{Conv}(\sX)$, the
\emph{local equilibrium measure}
\begin{equation}
  \label{eq:gibbs-gen}
  \vartheta^N_f(\eta) := \prod_{x \in \G_N}
  n_{\sigma\left(f_t\left(\frac{x}{N}\right)\right)}(\eta_x).
\end{equation}
At the macroscopic level, the counterpart of
$\mu^N_{\infty,\lambda}$ are the constant functions
$f \equiv \sigma^{-1}(\lambda)$, which are stationary solutions
to~\eqref{eq:limitPDE1}. Any equilibrium measure allows the
define the adjoint $\cL_N^*$ in
$L^2({\rm d}\vartheta_{f_\infty})$, acting on probability
densities of the form
$G^N_f := {\rm d} \vartheta_{f}^N/{\rm d} \vartheta_{f_\infty}$.

The two maps $\eta \mapsto \alpha_\eta ^N$ and
$f \mapsto \vartheta^N_f$ connect the functional spaces of the
microscopic and macroscopic evolutions, as summarized in
Figure~\ref{fig:setting}. The diagram only closes in the limit
$N \to \infty$ when the empirical measure becomes asymptotically
atomic, i.e. when correlations remain low and \emph{molecular
  chaos} holds.
\begin{figure}[!ht]
  \includegraphics[scale=0.9]{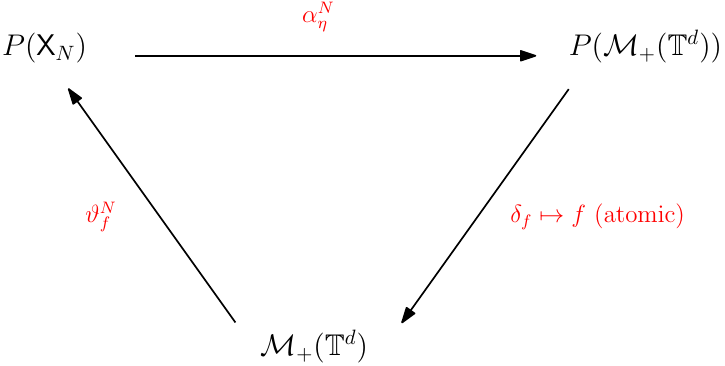}
  \caption{The functional setting. The downward right arrow
    simply embeds the atomic elements
    $\delta_f \in P(\mathcal M_+(\T^d))$ for some
    $f \in \mathcal M_+(\T^d)$ into $\mathcal M_+(\T^d)$ by the
    mapping $\delta_f \mapsto f$.}
  \label{fig:setting}
\end{figure}

In order to compare the two dynamics, we also define the adjoint
of $\cL_\infty$ the operator accounting for the time evolution of
$G^N_t := {\rm d} \vartheta_{f_t}^N/{\rm d}
\vartheta_{f_\infty}$, where $f_t$ evolves through
$\mathbf{L}_\infty $:
\begin{equation}
  \label{eq:LinfN}
  \cL_{\infty} ^* G^N_f := \left( \sum_{x \in \T^d_N} \frac{\sigma'\left(f\left(\frac{x}{N}\right)\right) \mathbf{L}_\infty f \left(\frac{x}{N}\right)}{n_{\sigma\left(f   \left(\frac{x}{N}\right)\right)}(\eta_x)} \Big[ \partial_\lambda n_\lambda(\eta_x) \Big]_{\big|    \lambda=\sigma\left(f\left(\frac{x}{N}\right)\right)} \right) G_f ^N.
\end{equation}

We denote by $\Lip(\sX_N)$ the set of Lipschitz functions
$\Phi : \sX_N \to \R$ with respect to the normalised $\ell^1$
norm, and we denote the corresponding semi-norm
\begin{equation*}
  [\Phi]_{\Lip(\sX_N)} := \inf_{\zeta, \eta \in \sX_N}
  \frac{|\Phi(\eta)- \Phi(\zeta)|}{\frac{1}{N^d} \sum_{x\in
    \T_N^d} |\eta_x-\zeta_x|}.
\end{equation*}
By duality we obtain the following Monge-Kantorovich distance over
$P(\sX_N)$
\begin{equation*}
  \forall \, \mu^N_1, \mu^N_2 \in P(\sX_N), \quad
  \left\| \mu^N_1 - \mu^N_2 \right\|_{\Lip^*} = \sup_{\Phi \in
    \Lip(\sX_N)^*} \frac{\left\langle
    \mu^N_1 - \mu^N_2, \Phi \right\rangle}{[\Phi]_{\Lip(\sX_N)}}.
\end{equation*}
where $\Lip(\sX_N)^*$ is the set of non-constant Lipschitz
functions.

Finally given $\mu, \nu \in P(X_N)$ we define
\begin{equation}
  \label{eq:rel_entropy}
  H^N(\mu | \nu) := \frac{1}{N^d} \int_{\sX_N}
  \log\big(\textstyle\frac{{\rm d}\mu}{{\rm d}\nu}\big) \dd \mu
  \in \R_+ \cap \{+\infty\}
\end{equation}
when $\mu$ has a Radon-Nikodym derivative with respect to $\nu$,
or simply $+\infty$ otherwise.

\subsection{Abstract assumptions}

We make the following assumptions
on~\eqref{eq:particle_evol}-\eqref{eq:limitPDE1}:
\medskip

\noindent
\textbf{(H1) Microscopic stability.} The microscopic evolution
semigroup $e^{t\mathcal{L}_N}$ satisfies
\begin{align}
  \label{eq:micro-stab}
  \forall \, \Phi \in \Lip(\sX_N), \ \forall \, t \ge 0, \quad
  \left[ e^{t\mathcal{L}_N}
  \Phi \right]_{\Lip(\sX_N) } \leq  C_S \left[  \Phi
  \right]_{\Lip(\sX_N)}
\end{align}
for some stability constant $C_S \ge 1$ (independent of time).
\medskip

\noindent
\textbf{(H2) Macroscopic stability.} There is a Banach space
$\fB \subset \cM_+(\G_\infty)$ so that~\eqref{eq:limitPDE1} is
locally well-posed in $\fB$. Moreover, given an initial condition
$f_0 \in \fB$, we write $f_t \in \fB$ the unique solution on its
maximal time interval of existence $[0,T_m)$, $f_\infty$ the
stationary constant solution with the same mass as $f_0$,
$G^N_t := {\rm d} \vartheta_{f_t^N}^N/{\rm d}
\vartheta_{f_\infty}$, and 
\begin{align}
  \label{eq:Rpar}
  & \forall \, t \in [0,T_m), \quad
    \mathsf R_1(t) := \left\| f_t - f_{\infty} \right\|_\fB \\
  \label{eq:lip-test}
  & \forall \, t \in [0,T_m), \quad
    \mathsf R_2(t) := \left[ \frac{\ln G^N_t}{N^d}
    \right]_{\Lip(\sX_N)} + \left[ \frac{\left[
    \cL_N+\cL_N^* \right] \sqrt{G^N_t}}{N^d
    \sqrt{G^N_t}} \right]_{\Lip(\sX_N)}. 
\end{align}
\medskip

\noindent
\textbf{(H3) Monge-Kantorovich Consistency.} There is a consistency
error $\epsilon_L(N) \ge 0$ so that, given any initial condition
$f_0 \in \fB$, on the maximal time of existence $t \in [0,T_m)$
of its corresponding solution $f_t \in \fB$, one has for any
$\Phi \in \Lip(\sX_N)$ with $[\Phi]_{\Lip(\sX_N)} \le 1$:
\begin{equation}
  \label{eq:consist}
  \int_0 ^t \int_{\sX_N} \left( e^{(t-\tau)\cL_N}
    \Phi \right) \left[ \cL_N^* G^N _\tau - \cL_\infty ^*
    G^N_\tau \right] \dd \vartheta^N_{f_\infty} \dd \tau
  \le \epsilon_L(N)  \int_0 ^t \mathsf R_1(\tau) \dd \tau.
\end{equation}

 \medskip

\noindent
\textbf{(H4) Entropic Consistency.} There is a consistency error
$\epsilon_{NL}(N) \ge 0$ so that, given any initial condition
$f_0 \in \fB$, on the maximal time of existence $t \in [0,T_m)$
of its corresponding solution $f_t \in \fB$, one has
\begin{align}
  \nonumber
  H^N\left( \vartheta_{f_0^N}^N | \vartheta^N _{f_\infty} \right) -
  H^N\left( \vartheta_{f_t^N}^N| \vartheta^N _{f_\infty} \right) 
  + \frac{2}{N^d} \int_0 ^t \int_{\sX_N}
  \sqrt{G^N_\tau}  \cL_N \left( \sqrt{G^N_\tau} \right) \dd
  \vartheta^N _{f_\infty} \dd \tau \\
  \label{eq:consist-ent}
  \le \epsilon_{NL}(N) \int_0 ^t \mathsf R_1(\tau) \dd \tau.
\end{align}

\subsection{The abstract approach}

\begin{theorem}[The abstract result]
  \label{thm:main}
  Consider an interacting particle
  system~\eqref{eq:particle_evol}-\eqref{eq:limitPDE1} with a
  product invariant measure as described, and satisfying {\bf
    (H1)-(H2)-(H3)}. Let $\mu_0^N \in P_1(X_N)$ for $N\geq 1$ and
  $f_0 ^N \in \mathcal{B}$. Let $f_t ^N \in \fB$ the unique
  solution to~\eqref{eq:limitPDE1} on the time of existence
  $[0,T_m)$ in~{\bf (H2)}. Given $T \in [0,T_m)$, one has
  \begin{align}
    \label{eq:cvg}
    & \sup_{t \in [0,T]} \left\| \mu_t ^N - \vartheta_{f_t ^N}^N
      \right\|_{\Lip^*} \le C_S \left\| \mu_0 ^N-
      \vartheta_{f_0 ^N} ^N \right\|_{\Lip^*} +  \epsilon_L(N) \int_0^T
      \mathsf R_1(t) \dd t.
  \end{align}

  If the interacting particle system additionally satisfies {\bf
    (H4)} then
  \begin{align}
    \label{eq:cvg-ent}
    & \sup_{t \in [0,T]} H^N \left(\mu_t ^N \vert \vartheta_{f_t ^N}^N
      \right) \le H^N \left( \mu_0 ^N \vert \vartheta_{f_0 ^N} ^N \right)
    \\
    \nonumber
    & \hspace{1cm} + \epsilon_L(N) \mathsf R_2(0) + 2 
      \epsilon_L(N) \left( \sup_{t \in [0,T]} \mathsf R_2(t)
      \right) \int_0^T \mathsf R_1(t) \dd t  + \epsilon_{NL}(N)
      \int_0^T \mathsf R_1(t) \dd t.
  \end{align}
\end{theorem}

\begin{proof}[Proof of Theorem~\ref{thm:main}]
  The proof is short and we present it immediately in order to
  provide the reader with a clear picture of the strategy. The
  rest of the paper is devoted to proving the assumptions on our
  particular model. Denote by
  $F^N_t := {\rm d} \mu^N_t/{\rm d}\vartheta^N_{f_\infty}$ and
  $G^N_t := {\rm d} \vartheta^N_{f_t ^N}/{\rm
    d}\vartheta^N_{f_\infty}$ the densities with respect to the
  equilibrium measure $\vartheta_{f_\infty}^N$.  
  
In order to prove \eqref{eq:cvg}, we write
\begin{align*}
  \partial_t \Big( F^N_t - G_t ^N \Big)
  & = \mathcal{L}_N^* \Big( F^N_t - G_t ^N \Big) + \left(
    \mathcal{L}_N ^* G_t ^N - \partial_t G_t ^N \right) \\
  & = \mathcal{L}_N^* \Big( F^N_t - G_t ^N \Big) + \left(
    \mathcal{L}_N ^* G_t ^N - \mathcal{L}_\infty^* G_t ^N \right)
\end{align*}
so that Duhamel's formula yields
\begin{align*}
  F^N_t - G_t ^N = e^{t \mathcal{L}_N ^*} \left( F^N_0 - G_0 ^N
  \right) + \int_0^t e^{(t-\tau)\mathcal{L}_N ^*} \left( \mathcal{L}_N ^*
  G_\tau ^N - \mathcal{L}_\infty^* G_\tau ^N \right) \dd \tau.
\end{align*}
Take $\Phi \in \Lip(\sX_N)$ with
$[ \Phi ]_{\operatorname{Lip}(\sX_N)} \leq 1$ and integrate the
above equation to get
\begin{align*}
  & \int_{\sX_N} \Phi  \Big( F^N_t - G_t ^N \Big) \dd   \vartheta^N_{f_\infty} \\
  & = \underbrace{\int_{\sX_N} \left( e^{t \mathcal{L}_N} \Phi \right) \Big( F^N_0 - G_0 ^N \Big) \dd \vartheta^N_{f_\infty}}_{I_1(t)} + \underbrace{\int_{\sX_N} \int_0^t \left( e^{(t-\tau) \mathcal{L}_N} \Phi  \right) \left( \mathcal{L}_N G_\tau ^N - \mathcal{L}_\infty G_\tau  ^N \right) \dd \vartheta^N_{f_\infty} \dd \tau}_{I_2(t)}.
\end{align*}
The assumption \textbf{(H1)} then implies
$I_1(t) \le C_S \| \mu^N_0 - \vartheta_{f_0} ^N \|_{\Lip^*}$ and
the assumption \textbf{(H3)} implies
$I_2(t) \le \epsilon_L(N) \int_0 ^t R(\tau) \dd \tau$ which
proves~\eqref{eq:cvg}. 

To prove the second
inequality~\eqref{eq:cvg-ent}, i.e. convergence in relative entropy, we write
\begin{align*}
  H^N \left(\mu_t ^N \vert \vartheta_{f_t ^N}^N \right)
  & = H^N \left(\mu_t ^N \vert \vartheta_{f_\infty}^N \right) -
     \int_{\sX_N} \frac{\ln G^N_t}{N^d} F^N_t \dd \vartheta_{f_\infty} ^N.
\end{align*}
We then use~\eqref{eq:cvg}, i.e. the hydrodynamic limit in dual Lipschitz distance, to estimate the second term of the right-hand side. The Lipschitz function needed for our purposes is given by $\frac{\ln G^N_t}{N^d}$ and the fact that this is indeed Lipschitz is shown in Section \ref{sec:entropy}. This yields
\begin{align*}
  - \int_{\sX_N} \frac{\ln G^N_t}{N^d} \dd \mu^N_t \le
  - \int_{\sX_N} \frac{\ln G^N_t}{N^d} G^N_t \dd
  \vartheta_{f_\infty} ^N + \epsilon_L(N) \mathsf R_2(t) \int_0
  ^t \mathsf R_1(\tau) \dd \tau.
\end{align*}
We deduce
\begin{align}
  \label{eq:ent1}
  H^N \left(\mu_t ^N \vert \vartheta_{f_t ^N}^N \right)
  & \le H^N \left(\mu_t ^N \vert \vartheta_{f_\infty}^N \right) -
    H^N \left(\vartheta_{f_t ^N} ^N \vert \vartheta_{f_\infty}^N
    \right) + \epsilon_L(N) \mathsf R_2(t) \int_0 ^t
    \mathsf R_1(\tau) \dd \tau.
\end{align}
We now connect the entropies at time $t$ on the right-hand side
to the initial entropies using respectively the evolution
equation of the particle system and the entropic
consistency \textbf{(H4)}, ~\eqref{eq:consist-ent}:
\begin{align}
  \label{eq:ent2}
  & H^N \left(\mu_t ^N \vert \vartheta_{f_\infty}^N \right) = H^N
  \left(\mu_0 ^N \vert \vartheta_{f_\infty}^N \right) + \frac{1}{N^d} \int_0 ^t
  \int_{\sX_N} F^N_\tau \cL_N \left( \ln F^N_\tau \right) \dd
    \vartheta_{f_\infty} ^N \dd\tau \\
  \label{eq:ent3}
  & - H^N \left(\vartheta_{f_t^N} ^N \vert \vartheta_{f_\infty}^N
    \right) \le - H^N\left( \vartheta_{f_0}^N \vert \vartheta^N
    _{f_\infty} \right) - \frac{2}{N^d} \int_0 ^t \int_{\sX_N}
    \sqrt{G^N_\tau}  \cL_N \left( \sqrt{G^N_\tau} \right) \dd
    \vartheta_{f_\infty} ^N \dd \tau \\
  \nonumber 
  & \hspace{7cm} + \epsilon_{NL}(N) \int_0 ^t \mathsf
    R_1(\tau) \dd \tau.
\end{align}
We then note: 
\begin{enumerate} 
\item Whenever $\cL_N$ is the
generator of a Markov process, we have that 
\begin{align*}
  -  2 \int_0 ^t \int_{\sX_N}
  \sqrt{F^N_\tau}  \cL_N \left( \sqrt{F^N_\tau} \right) \dd
  \vartheta_{f_\infty} ^N \dd \tau \le - \int_0 ^t \int_{\sX_N}
  F^N_\tau \cL_N \ln F^N_\tau \dd\vartheta_{f_\infty} ^N \dd \tau, 
\end{align*} 
\item and also
\begin{align}
  \label{eq:cont-var-ent}
  - \int_0 ^t \int_{\sX_N} \frac{\left[\cL_N + \cL_N^*\right]\left(
  \sqrt{G^N_\tau}  \right)}{\sqrt{G^N_\tau}} F^N_\tau \dd
  \vartheta_{f_\infty} ^N \dd \tau \le - 2 \int_0 ^t \int_{\sX_N}
  \sqrt{F^N_\tau}  \cL_N \left(\sqrt{F^N_\tau} \right) \dd
  \vartheta_{f_\infty} ^N \dd \tau.
\end{align}
\end{enumerate} 
The point (1) follows from the inequality 
$(\ln b - \ln a) a \le 2 (\sqrt{b}-\sqrt{a})\sqrt{a}$, that holds for all
$a,b >0$. 
The point (2), is seen as follows: 
If $P_t := \exp(t [\cL_N+\cL_N^*])$,  then the transition
kernel of this semigroup satisfies
$p_t(\eta,\zeta)=p_t(\zeta,\eta)$ and for any $F,G > 0$ (the
non-negative case follows by approximation)
\begin{align*}
  &  \int_{\sX_N} \frac{P_t \left( \sqrt{G}
    \right)}{\sqrt{G}} F \dd \vartheta_{f_\infty} ^N \\
  & = \frac12 \int_{\sX_N \times \sX_N} p_t(\eta,\zeta) \left( 
    \frac{\sqrt{G(\zeta)F(\eta)}}{\sqrt{G(\eta)F(\zeta)}}
    + \frac{\sqrt{G(\eta)F(\zeta)}}{\sqrt{G(\zeta)F(\eta)}} \right)
    \sqrt{F(\eta)} \sqrt{F(\zeta)}
    \dd \vartheta_{f_\infty}^N (\eta) \dd \vartheta_{f_\infty} ^N (\zeta)
  \\
  & \ge \int_{\sX_N \times \sX_N} p_t(\eta,\zeta) 
    \sqrt{F(\eta)} \sqrt{F(\zeta)}
    \dd \vartheta_{f_\infty} ^N (\eta) \dd \vartheta_{f_\infty}
    ^N (\zeta) = \int_{\sX_N} \left( P_t\sqrt{F} \right)
    \sqrt{F}\dd \vartheta_{f_\infty} ^N.
\end{align*}
Since both quantities agree at $t=0$ by differentiating at
$t=0$ we get
\begin{align*}
  \int_{\sX_N} \sqrt{F^N_\tau}
  \left[ \cL_N + \cL_N^* \right] \left(\sqrt{F^N_\tau} \right)
  \dd \vartheta_{f_\infty} ^N \le \int_{\sX_N} \frac{\left[\cL_N +
  \cL_N^*\right]\left( \sqrt{G^N_\tau} \right)}{\sqrt{G^N_\tau}}
  F^N_\tau \dd \vartheta_{f_\infty} ^N.
\end{align*}
This implies indeed~\eqref{eq:cont-var-ent} using the symmetry of the
left-hand side. 

We now apply the hydrodynamic limit in the dual-Lipschitz 
norm, in \eqref{eq:cvg}, to the left-hand side
of \eqref{eq:cont-var-ent}.  
The Lipschitz function needed for our purposes is 
$$(G^N_\tau)^{-1/2} \left[ \cL_N + \cL_N^* \right] \left(
  \sqrt{G^N_\tau} \right)$$ (using the Lipschitz bound in {\bf
  (H2)}).  This yields
\begin{align*}
  - \int_0 ^t \int_{\sX_N}
  \frac{\left[ \cL_N + \cL_N ^* \right]\left( \sqrt{G^N_\tau}
  \right)}{\sqrt{G^N_\tau}} G^N_\tau \dd \vartheta_{f_\infty}
  ^N\dd \tau \le - \int_0 ^t \int_{\sX_N}
  \frac{\left[ \cL_N + \cL_N^* \right] \left( \sqrt{G^N_\tau}
  \right)}{\sqrt{G^N_\tau}} F^N_\tau \dd
  \vartheta_{f_\infty} ^N \dd \tau \\
  + N^d \epsilon_L(N) \left( \sup_{\tau \in [0,t]} \mathsf R_2(\tau) \right)
  \int_0 ^t \mathsf  R_1(\tau) \dd \tau.
\end{align*}
   The fact that this serves as such test function can be
explicitly verified: in the ZRP case the specific calculations
 are performed in Section \ref{sec:entropy}. 

We now continue by using that
\begin{align*}
  - 2 \int_{\sX_N} \sqrt{G^N_\tau} \cL_N
  \left( \sqrt{G^N_\tau} \right) \dd \vartheta_{f_\infty} ^N =
  - \int_{\sX_N} \frac{\left[ \cL_N + \cL_N ^* \right]\left(
  \sqrt{G^N_\tau} \right)}{\sqrt{G^N_\tau}} G^N_\tau \dd
  \vartheta_{f_\infty} ^N,
\end{align*}
we thus deduce
\begin{align*}
  - \frac{2}{N^d} \int_{\sX_N} \sqrt{G^N_\tau} \cL_N
  \left( \sqrt{G^N_\tau} \right) \dd \vartheta_{f_\infty} ^N \le
  - \frac{1}{N^d} \int_0 ^t \int_{\sX_N} F^N_\tau \cL_N \ln F^N_\tau
  \dd\vartheta_{f_\infty} ^N \dd \tau \\
  + \epsilon_L(N) \left(\sup_{\tau \in [0,t]} \mathsf R_2(\tau)
  \right) \int_0 ^t \mathsf  R_1(\tau) \dd \tau.
\end{align*}
By plugging this last estimate
into~\eqref{eq:ent1}-\eqref{eq:ent2}-\eqref{eq:ent3} we obtain
\begin{align*}
  & H^N \left(\mu_t ^N \vert \vartheta_{f_t ^N}^N \right)
  \le H^N \left(\mu_0 ^N \vert \vartheta_{f_\infty}^N \right) -
  H^N \left(\vartheta_{f_0^N} ^N \vert \vartheta_{f_\infty}^N \right)
  \\
  & \hspace{3cm}
    + \epsilon_{NL}(N) \int_0 ^t \mathsf R_1(\tau) \dd \tau +
  \epsilon_L(N) \left( \sup_{\tau \in [0,t]} \mathsf R_2(\tau)
  \right) \int_0 ^t \mathsf  R_1(\tau) \dd \tau.
\end{align*}
and finally (using the linear consistency again at time $t=0$)
\begin{align*}
  H^N \left(\mu_t ^N \vert \vartheta_{f_t^N}^N \right) 
  \le H^N
  \left(\mu_0 ^N \vert \vartheta_{f_0^N}^N \right) 
  + \epsilon_L(N)
  R_2(0) + \epsilon_{NL}(N) \int_0 ^t \mathsf R_1(\tau) \dd \tau
  \\ + 2 \epsilon_L(N) \left( \sup_{\tau \in [0,t]} \mathsf
  R_2(\tau) \right) \int_0 ^t \mathsf R_1(\tau) \dd \tau
\end{align*}
which concludes the proof.
\end{proof}

\begin{remark}
  Note that $\| \mu^N - \vartheta_f^N \|_{\Lip^*} \to 0$ as
  $N \to \infty$ implies that the empirical
  measure~\eqref{eq:emp_meas_dirac} sampled from the law
  $\mu^N$ satisfies
  \begin{equation}
    \label{eq:hydro_profile}
    \forall \, \phi \in C^1(\T^d), \ \forall \, \epsilon >0, \
    \forall \, t \ge 0,\quad
    \lim_{N\rightarrow\infty} \mu^N\left( \left\{ |\langle
        \alpha^N_\eta, \phi \rangle - \langle f, \phi \rangle
        | > \epsilon \right\} \right) = 0
  \end{equation}
  which can be quantified (thus recovering quantitatively
  the type of convergence from~\cite{GPV88}):
  \begin{align*}
    & \mu^N\left( \left\{ |\langle
      \alpha^N_\eta, \phi \rangle - \langle f, \phi \rangle
      | > \epsilon \right\} \right) \\ 
    & \le \mu^N\left( \left\{ \langle
      \alpha^N_\eta, \phi \rangle \ge \langle f, \phi \rangle
      + \epsilon \right\} \right) + \mu^N\left( \left\{ \langle
      \alpha^N_\eta, \phi \rangle \le \langle f, \phi \rangle
      - \epsilon \right\} \right) \\
    & \le \int_{\sX_N} \left[ F_\epsilon ^+ \left( \left\langle \phi,
      \alpha_\eta ^N \right\rangle \right)  -
      F_\epsilon ^+  \left( \left\langle \phi, f \right\rangle
      \right) + F_\epsilon ^-
      \left( \left\langle \phi, \alpha_\eta ^N
      \right\rangle \right)  - F^-_\epsilon  \left( \left\langle \phi,
      f \right\rangle \right) \right] \dd \mu_t^N \\
    & \le \int_{\sX_N} \int_{\sX_N} \left[ F_\epsilon ^+ \left(
      \left\langle \phi,  \alpha_\eta ^N \right\rangle \right)  -
      F_\epsilon ^+  \left( \left\langle \phi, \alpha_\zeta ^N \right\rangle
      \right) \right] \dd \mu^N(\eta) \dd \vartheta_f ^N(\zeta) \\
    & \qquad + \int_{\sX_N} \int_{\sX_N} \left[ F_\epsilon^-
      \left( \left\langle \phi, \alpha_\eta ^N
      \right\rangle \right)  - F^-_\epsilon  \left( \left\langle \phi,
      \alpha_\zeta ^N \right\rangle \right) \right] \dd \mu^N(\eta) \dd
      \vartheta_f ^N(\zeta) \\
    & \qquad + \int_{\sX_N} \left[ F^+_\epsilon  \left(
      \left\langle \phi, \alpha_\zeta ^N \right\rangle \right) -
      F^+_\epsilon  \left( \left\langle \phi,
      f \right\rangle \right) + F^-_\epsilon  \left( \left\langle \phi,
      \alpha_\zeta ^N \right\rangle \right) - F^-_\epsilon
      \left( \left\langle \phi, f \right\rangle \right) \right]
      \dd \vartheta_f ^N(\zeta) 
  \end{align*}
  where $F_\epsilon ^\pm$ are mollified characteristic functions
  of respectively
  $\{ z \ge \left\langle \phi, f \right\rangle + \epsilon \}$ and
  $\{ z \le \left\langle \phi, f \right\rangle - \epsilon \}$,
  and we deduce the estimate
  \begin{equation}
    \label{eq:cvg-proba}
    \mu^N\left( \left\{ |\langle
        \alpha^N_\eta, \phi \rangle - \langle f, \phi \rangle
        | > \epsilon \right\} \right) \lesssim \epsilon^{-1}
    \left( \| \mu^N - \vartheta_f ^N \|_{\Lip^*} +
      N^{-\frac{d}{2}} + \mathsf R_3(N) \right)
  \end{equation}
  where $\mathsf R_3(N)$ is the convergence rate of the Riemann
  sum for $\phi f$. This follows from
  $\| F_\epsilon ^{\pm} \|_{\Lip} \lesssim \epsilon^{-1}$, the Cauchy-Schwarz inequality and the following estimate:
  \begin{align*}
    & \int_{\sX_N} \left| \left\langle \phi, \alpha_\zeta^N      \right\rangle - \left\langle \phi, f \right\rangle     \right|^2 \dd \vartheta_{f} ^N(\zeta) = \int_{\sX_N} \left| N^{-d} \sum \phi\left(\frac{x}{N}\right) \zeta_x - \int_{\T^d} \phi f   \right|^2 \dd \vartheta_{f} ^N(\zeta) \\
    & \lesssim \left( N^{-d} \sum_{x \in \T^d_N}
      \phi\left(\frac{x}{N}\right) f\left(\frac{x}{N}\right) -
      \int_{\T^d} \phi f \right)^2 + \int_{\sX_N} \left( N^{-d} \sum_{x \in \T^d_N} \phi\left(\frac{x}{N}\right) \left( \zeta_x - f\left(\frac{x}{N}\right) \right)  \right)^2 \dd \vartheta_f ^N\\
    & \lesssim \mathsf R_3(N)^2 + N^{-d}.
  \end{align*}
  The combination of~\eqref{eq:cvg-proba} and~\eqref{eq:cvg}
  provides a quantitative rate in~\eqref{eq:hydro_profile}.
\end{remark}

\section{The Zero-Range Process}


The state space is $\sX=\N$, meaning that the number of
particles is not limited at each site. Given a \emph{transition
  function} $p \in P(\T^d_N)$ with $p(0)=0$ and bounded second moment, and a \emph{jump rate
  function} $g : \N \rightarrow \R_+$, the base generator
$\hat \cL_N$ writes
\begin{equation}
  \label{eq:zrp-gen}
  \forall \, \Phi \in C_b(\sX_N), \ \forall \, \eta \in \sX_N, \quad
  \hat \cL_N \Phi(\eta) := \sum_{x,y \in \T^d_N} p(y-x) g(\eta_x)
  \left[\Phi(\eta^{xy}) - \Phi(\eta) \right]
\end{equation}
where $\eta^{xy}$ is defined by
\begin{equation}
  \label{eq:def-etaxy}
  \eta^{xy}_z =
  \begin{cases}
    \eta_z \quad \text{ if } z \not \in \{x,y\} \\
    \eta_x-1 \quad \text{ if } z=x \\
    \eta_y+1 \quad \text{ if } z=y.
  \end{cases}
\end{equation}
The local
equilibrium structure is given by
\begin{align}
  \label{eq:zrp-gibbs}
  & n_\lambda(k):= \frac{\lambda^k}{g(k)!
    Z(\lambda)} \quad \text{ with } \quad
    Z(\lambda) := \sum_{k=0} ^{+\infty} \frac{\lambda^k}{g(k)!}\\
  & \sigma \ \text{ is defined implicitely by } \ \sigma(\rho)
    \frac{Z'(\sigma(\rho))}{Z(\sigma(\rho))} \equiv \rho
\end{align}
with the notation $g(k)!:=g(k) g(k-1) \cdots g(1)$ and
$g(0)!=1$. The pair $(n_\lambda,\sigma)$ satisfies
\begin{equation*}
  \begin{cases}
    \ds \mathbb E_{n_{\sigma(\rho)}}[k] = \sum_{k \ge 0} \frac{k
      \sigma(\rho)^k}{g(k)! Z(\sigma(\rho))} = \rho, \\[2mm]
    \ds \mathbb E_{n_{\sigma(\rho)}}[g(k)]= \sum_{k \ge 0}
    \frac{\sigma(\rho)^k}{g(k-1)! Z(\sigma(\rho))} =
    \sigma(\rho).
  \end{cases}
\end{equation*}
When $f \equiv \rho \in [0,+\infty)$ is constant, the local
equilibrium measure
$\vartheta_{\rho}^N=\mu_{\infty,\sigma(\rho)}^N$
in~\eqref{eq:inv-meas-gen}-\eqref{eq:gibbs-gen} is invariant with
average number of particles $\rho$. The \emph{mean
  transition rate} is defined by
\begin{equation}
  \label{eq:transition}
  \gamma := \sum_{x \in \Z^d} x p(x) \in \R^d.
\end{equation}
 When
$\gamma \not =0$, the first non-trivial scaling limit is the
hyperbolic scaling $\cL_N := N \hat \cL_N$, and the expected
limit equation is
\begin{equation}
  \label{eq:zrp-hyp}
  \partial_t f = \gamma \cdot \nabla [\sigma(f)] \quad \text{
    with $\gamma$ defined in~\eqref{eq:transition}.}
\end{equation}
We assume in this paper that $\gamma =0$, then the first
non-trivial scaling limit is the parabolic scaling
$\cL_N := N^2 \hat \cL_N$, and the expected limit equation is
\begin{equation}
  \label{eq:zrp-diff}
  \partial_t f = \mathtt A : \nabla^2 \sigma(f) \quad \text{ with } \quad
  \mathtt A = \left[ (a_{ij})_{i,j=1} ^d \right] \quad \text{ and } \quad a_{ij} := \frac12 \sum_{x \in \Z^d} p(x) x_i x_j.
\end{equation}

The main result is:
\begin{theorem}[Hydrodynamic limit for the one and two-dimensional
  centred ZRP]
  \label{theo:ZRP}
  \mbox{ }

  Consider $\hat \cL_N$ defined in~\eqref{eq:zrp-gen} with $d=1$ or $d=2$, with $\gamma=0$, and where the jump rate jump rate $g : \R_+ \to [0,+\infty)$ satisfies $g(0) = 0$, $g(n)>0$ for all $n>0$, $g$ non-decreasing and uniformly Lipschitz on $\R_+$, and finally there are $n_0 > 0$ and $c_g > 0$ such that $g(n')-g(n) \ge c_g$ for any $n'\ge n+n_0$. 

  Consider $\mu_0^N \in P_1(\sX_N)$ for all $N \ge 1$, where $\sX_N := \N^{\T_N}$. Define $\mu^N_t := e^{t \cL_N} \mu^N_0$ with $\cL_N:=N^2\hat \cL_N$ and, given $f_0 \in C^3(\T)$ with $f_0 \ge \delta>0$, consider $f_t \in C([0,T],C^3(\T^d))$ solution to~\eqref{eq:zrp-diff}. Then there is $C > 0$ so that for any $\varepsilon_0 > 0$ it holds
  \begin{align}
    \label{eq:dec-zrp-diff}
    & \sup_{t \in [0,T]} \left\| \mu_t ^N - \vartheta_{f_t}^N
    \right\|_{\Lip^*}  \le  \left\| \mu_0 ^N-
      \vartheta_{f_0} ^N \right\|_{\Lip^*} + 
      \begin{cases} 
      C N^{-\frac{1}{6}+ \varepsilon_0}, \quad \text{ if } d=1\\[2mm]
      C[ \ln (N) ]^{-\frac18 +\varepsilon_0}, \quad \text{ if } d=2, 
      \end{cases}   \\ \nonumber 
      \vspace{0.001cm}\\
    \label{eq:dec-zrp-diff-ent}
    & \sup_{t \in [0,T]} H^N \left( \mu_t ^N | \vartheta_{f_t}^N
    \right) \le 
    H^N \left( \mu_0 ^N | \vartheta_{f_0} ^N \right)
    + \begin{cases} 
      C N^{-\frac{1}{6}+ \varepsilon_0}, \quad \text{ if } d=1\\[2mm]
      C[ \ln (N) ]^{-\frac18 +\varepsilon_0}, \quad \text{ if } d=2.
      \end{cases}
  \end{align}
  \end{theorem}

  The rest of the article is devoted to proving the abstract assumptions {\bf (H1)}--{\bf (H2)}--{\bf (H3)}--{\bf (H4)} for the ZRP in order to apply Theorem~\ref{thm:main}. Note that, although we prove the general assumptions  for the zero-range process, it is not hard to verify that these assumptions also hold for other prototypical symmetric and attractive models, such as the Simple Exclusion Process on $\{0,1\}^{\mathbb{T}_N}$. 

\section{Stability properties}
\label{sec:zrp}



\subsection{Microscopic stability on average}
\label{subsec:micro1}

We use the standard coupling as in~\cite{Liggett1985}. Let us define the following coupling operator
\begin{align}
  \nonumber
  \widetilde{\mathcal{L}}_N\Psi(\eta, \zeta) :=
  & \sum_{x,y \in
    \T^d_N} p(y-x) \Big( g(\eta_x) \wedge g(\zeta_x) \Big)
    \Big[ \Psi(\eta^{xy}, \zeta^{xy}) -
    \Psi(\eta,\zeta) \Big] \\
  \nonumber 
  & + \sum_{x,y \in \T^d_N} p(y-x) \Big(g(\eta_x) - g(\eta_x)
    \wedge g(\zeta_x) \Big)
    \Big[ \Psi(\eta^{xy}, \zeta) - \Psi(\eta, \zeta) \Big] \\
  \label{def:coupling-zrp}
  & + \sum_{x,y \in \T^d_N} p(y-x) \Big( g(\zeta_x) - g(\eta_x)
    \wedge g(\zeta_x) \Big) \Big[ \Psi(\eta,\zeta^{xy}) -
    \Psi(\eta, \zeta) \Big]
\end{align}
for a two-variable test function $\Psi(\eta,\zeta)$. Then we have
the following three properties:
\begin{itemize}
\item[\textbf{(P1)}]
  $\widetilde{\mathcal{L}}_N\Phi(\eta) = \hat \cL_N \Phi(\eta)$ and
  $\widetilde{\mathcal{L}}_N\Phi(\zeta) = \hat \cL_N
  \Phi(\zeta)$ for a function $\Phi$ of only one variable.
\item[\textbf{(P2)}] The semigroup $e^{t\widetilde{\mathcal{L}}_N}$
  preserves the sign (like all generators of Markov processes). 
\item[\textbf{(P3)}] The coupling operator $\widetilde{\mathcal{L}}_N$ satisfies the inequality
  \begin{align}
    \label{eq:P3}
    \widetilde{\mathcal{L}}_N \left( \sum_{z \in \mathbb{T}_N^d}
    \vert \eta_z- \zeta_z\vert \right) \leq 0.
  \end{align}
\end{itemize}

\textbf{(P1)} and \textbf{(P2)} are trivial, and we prove \textbf{(P3)} below. Let us first explain how to deduce the stability from
\textbf{(P1)}-\textbf{(P2)}-\textbf{(P3)}. \textbf{(P1)} implies that
\begin{align*}
  \Phi_t(\eta) - \Phi_t(\zeta) = e^{t\mathcal{L}_N} \Phi_0 (\eta)
  - e^{t\mathcal{L}_N} \Phi_0(\zeta) =
  e^{t\widetilde{\mathcal{L}}_N} \Psi_0(\eta,\zeta) =:
  \Psi_t(\eta,\zeta)
\end{align*}
with $\Psi_0(\eta,\zeta) := \Phi_0(\eta) - \Phi_0(\zeta)$. If we
assume that the Lipschitz bound holds initially for some constant
$C \ge 0$, then
\begin{align*}
  - \frac{C}{N^d} \left( \sum_{z \in \mathbb{T}_N^d}
  \vert \eta_z- \zeta_z\vert \right) \le \Psi_0(\eta,\zeta) \le
  \frac{C}{N^d} \left( \sum_{z \in \mathbb{T}_N^d} \vert \eta_z-
  \zeta_z\vert \right)
\end{align*}
\textbf{(P2)} then implies
\begin{align*}
  - \frac{C}{N^d} e^{t\widetilde{\mathcal{L}}_N} \left( \sum_{z
  \in \mathbb{T}_N^d} \vert \eta_z- \zeta_z\vert \right) \le
  \Psi_t(\eta,\zeta) \le \frac{C}{N^d}
  e^{t\widetilde{\mathcal{L}}_N}\left( \sum_{z \in
  \mathbb{T}_N^d} \vert \eta_z- \zeta_z\vert \right)
\end{align*}
and finally \textbf{(P3)} implies
\begin{align*}
  e^{t\widetilde{\mathcal{L}}_N}\left( \sum_{z \in
  \mathbb{T}_N^d} \vert \eta_z- \zeta_z\vert \right) \le \left(
  \sum_{z \in \mathbb{T}_N^d} \vert \eta_z- \zeta_z\vert \right)
\end{align*}
which concludes the proof of~\eqref{eq:micro-stab} with $C_S=1$. To prove the crucial property \textbf{(P3)} (equation~\eqref{eq:P3}), we compute
\begin{align*} 
  \widetilde{\mathcal{L}}_N \left( \sum_{z \in \T^d_N} |\eta_z
  - \zeta_z| \right)
  & = \sum_{x,y \in \T^d_N} p(y-x) \Big( g(\eta_x) - g(\eta_x) \wedge
    g(\zeta_x) \Big) \\
  & \hspace{2cm} \times \Big[ |\eta^{xy}_x- \zeta_x| +
    |\eta^{xy}_y- \zeta_y| - |\eta_x- \zeta_x| -
    |\eta_y-\zeta_y| \Big] \\
  & \quad + \sum_{x,y \in \T^d_N} p(y-x) \Big( g(\zeta_x) - g(\eta_x)
    \wedge g(\zeta_x) \Big) \\
  & \hspace{2cm} \times \Big[ |\eta_x- \zeta^{xy}_x| +
    |\eta_y - \zeta^{xy}_y| - |\eta_x - \zeta_x| -
    |\eta_y -\zeta_y| \Big].
\end{align*}
When $g(\eta_x) - g(\eta_x) \wedge g(\zeta_x) >0$ necessarily
$\eta_x- \zeta_x \ge 1$ and
\begin{equation*}
  \Big[ |\eta^{xy}_x- \zeta_x| + |\eta^{xy}_y- \zeta_y| -
  |\eta_x- \zeta_x| - |\eta_y-\zeta_y| \Big] \le 0.
\end{equation*}
When $g(\zeta_x) - g(\eta_x) \wedge g(\zeta_x) >0$ necessarily
$\zeta_x - \eta_x \ge 1$ and
\begin{equation*}
  \Big[ |\eta_x- \zeta^{xy}_x| + |\eta_y- \zeta^{xy}_y| -
  |\eta_x- \zeta_x| - |\eta_y-\zeta_y| \Big] \le 0.
\end{equation*}
This proves~\eqref{eq:micro-stab} with $C_S=1$.

\begin{remark}
  The microscopic stability property is valid in any dimension $d \ge 1$. 
\end{remark}

\subsection{Microscopic stability in support}
\label{subsec:micro2}

In fact, with a minor modification of the previous argument, it is possible to strengthen the conclusion and obtain the stability at the level of the support. We record this property in the following lemma, which will prove useful in the consistency estimate. 
\begin{lemma}
  \label{lem:jumpdistance estim}
  Let $\eta^t$ and $\zeta^t$ be two attractive zero-range processes (i.e. with non-decreasing jump rate $g$) coupled along the standard coupling, then
  \begin{equation}\label{eq: support stability}
    \| \eta^t - \zeta^t \|_{\ell^1} := \sum_{x \in \T^d} \left|\eta_x^t - \zeta_x^t\right| \le \| \eta^0 - \zeta^0 \|_{\ell^1} := \sum_{x \in \T^d} \left|\eta_x^0 - \zeta_x^0\right|.
  \end{equation}
\end{lemma}
\begin{proof}[Proof of Lemma \ref{lem:jumpdistance estim}]
  Given $K>0$, whenever $g(\eta_z) > g(\zeta_z)$ (which implies $\eta_z > \zeta_z$),
  \begin{align*}
    &\Bigg[\left( \sum_{x \neq z, z'} |\eta_x - \zeta_x| - K\right) + |\eta_z - \zeta_z - 1| + |\eta_{z'} - \zeta_{z'} + 1|\Bigg] - \left[  \sum_{x} |\eta_x - \zeta_x| - K \right] \\ 
    &\leq \left( \sum_{x \neq z, z'} |\eta_x - \zeta_x| - K \right) + |\eta_z - \zeta_z| - 1 + |\eta_{z'} - \zeta_{z'}| + 1 - \left[  \sum_{x} |\eta_x - \zeta_x| - K \right] = 0.
  \end{align*}
  This implies
  $$ \left(\sum_{x} |\eta_x^{zz'} - \zeta_x| - K  \right)_+ \leq \left( \sum_{x} |\eta_x - \zeta_x| - K \right)_+$$
  and by summation
 $$ \sum_{zz'}p(z' - z) (g(\eta_z) - g(\zeta_z))_+ \left[\left(\sum_{x} |\eta_x^{zz'} - \zeta_x| - K  \right)_+ - \left( \sum_{x} |\eta_x - \zeta_x| - K \right)_+ \right] \leq 0.$$
 
A symmetric inequality holds whenever $g(\zeta_z) > g(\eta_z)$, and we conclude  that
 $$\widetilde{\mathcal{L}} \left( \left(\sum_z |\eta_z - \zeta_z| - K \right)_+ \right) \leq 0$$ (the part of the coupling generator where jumps occur simultaneously does not contribute trivially). We deduce that the coupling measure $\Pi_t$ satisfies 
 $$\frac{\dd}{\dd t} \left( \iint \left( \sum_z |\eta_z - \zeta_z| - K \right)_+ \dd \pi_t^N(\eta, \zeta) \right) \leq 0.$$
 Choosing $K=\|\eta^0-\zeta^0\|_{\ell^1}$, we conclude $\operatorname{supp}(\pi_t^N) \subset \{ \| \eta - \zeta\|_{\ell^1} \leq K \}$ for all times. 
\end{proof}

\subsection{Macroscopic stability}

In the parabolic scaling the limit PDE is the nonlinear diffusion equation~\eqref{eq:zrp-diff}. It is proved in~\cite[Chapter~2, Section~3]{KL99} that our assumptions on $g$ imply that $Z$ defined in~\eqref{eq:zrp-gibbs} has infinite radius of convergence, that $R(\lambda) = \lambda \partial_\lambda \ln(Z(\lambda)) = \frac{1}{Z(\lambda)} \sum_{n\geq0} \frac{n\lambda^n}{g(n)!}$ is well-defined on $\R_+$ and strictly increasing, and its inverse $\sigma=R^{-1}$ is smooth with all its derivatives uniformly bounded on $\R_+$ and $\sigma'(\R_+) \subset [\Lambda,\Lambda^{-1}]$ with $\Lambda \in (0,1)$. Hence by first applying the De Giorgi-Nash $L^\infty \cap C^\alpha$ estimate and second iterating Schauder estimates, we prove that an initial data in $C^3(\T^d)$ gives rise to a global solution with uniform in time bounds in $L^\infty \cap C^3(\T^d)$. The condition $f_t \ge \delta$ is propagated in time by maximum principle. Finally $f_t \to \rho$ exponentially fast as $t \to \infty$ in $C^\infty$ by parabolic regularization and $L^2$ spectral gap.

\begin{remark}
  This macroscopic stability is valid in any dimension $d \ge 1$.
\end{remark}

\section{Monge-Kantorovich consistency}

\subsection{Preliminary on the local Gibbs measure}

Given $f_t \in C^{3}(\T^d)$ with $f> \delta$, $\delta>0$, and $\rho := \int_{\T^d} f$, the density of the local Gibbs measure relatively to the invariant measure with mass $\rho$ is given as follows for the ZRP:
\begin{equation}
  \label{eq:zrp-psi}
  G_t^N (\eta) := \frac{{\rm d} \vartheta_f^N(\eta)}{{\rm d}
    \vartheta_{\rho}^N(\eta)} = \prod_{x \in \T^d_N}
  \left(\frac{\sigma \left(f_t \left( \frac{x}{N} \right)
      \right)}{\sigma(\rho)} \right)^{\eta_x}
  \left(\frac{ Z(\sigma \left(f_t \left( \frac{x}{N} \right)
      \right))}{Z(\sigma(\rho))} \right)^{-1}.
\end{equation}
The phase space $\sX$ is unbounded but the
invariant measures have exponential moments,
see~\cite[Corollary~3.6]{KL99}.

\subsection{The error estimate}

\begin{proposition}[Consistency in weak measure distance]
  \label{prop:zrp-consist-tspt}
  Assume $d=1$ or $d=2$, and consider a solution on
  $[0,T]$ to \eqref{eq:zrp-diff} satisfying 
  $f_t \in C^{3}(\T)$ with
    $f_t \ge \delta$ for some $\delta >0$ and $G_t ^N$ defined in~\eqref{eq:zrp-psi}.
  Then for every $\Phi \in \operatorname{Lip}(\sX_N)$ and every $\varepsilon_0>0$: 
    \begin{align*}
  \begin{cases} \ds
    I_t ^N := \int_0 ^t \int_{\sX_N} \left( e^{(t-\tau)\cL_N^*} \Phi    \right) \left[ \cL_N ^* G^N_\tau  - \cL_\infty ^* G^N_\tau \right] \dd \vartheta_{f_\infty} ^N \dd \tau = \cO \left( N^{-\frac{1}{6} + \varepsilon_0} \right), \ \text{ when } d=1
    \\[3mm] \ds
    I_t ^N := \int_0 ^t \int_{\sX_N} \left( e^{(t-\tau)\cL_N^*} \Phi    \right) \left[ \cL_N ^* G^N_\tau  - \cL_\infty ^* G^N_\tau \right] \dd \vartheta_{f_\infty} ^N \dd \tau =   \mathcal{O} \left( [ \ln (N) ]^{-\frac18} \right), \ \text{ when } d=2, 
  \end{cases}
  \end{align*}
  where the constant depends on the $C^{3}(\T^d)$ bound on $f_t$.
\end{proposition}
 

 
\subsection{Structure of the proof}

We break the proof into six steps.
\begin{enumerate}
\item Site-by-site error calculation.
\item Introduction of an intermediate scale.
\item Projection on the local equilibrium.
\item Local entropy production estimate.
\item Local limit theorem at the local equilibrium.
\item Synthesis of the error terms.
\end{enumerate}

\subsection{Step 1. Site-by-site error calculation}

We compute 
\begin{equation*}
  \cL_N ^*G^N_\tau  - \cL_\infty^* G^N_\tau =
  \left( \sum_{x \in \T^d_N} A_{\tau,x} ^N \right) G^N_\tau
\end{equation*}
where the site-by-site consistency coefficient $A_{\tau,x}^N$
depends on the model. 
We denote, given a time-dependent function
$f$ on the torus $\T^d$, $f_{\tau,x} := f_\tau(x/N)=f(\tau,x/N)$.
For the specific ZRP model of this proof, it is given by
\begin{align*}
  A_{\tau,x} ^N
  & := N^2 \sum_{y \in \T^d_N} p(y-x) g(\eta_x) \left(
    \frac{\sigma(f)_{\tau,y}}{\sigma(f)_{\tau,x}} -1 \right) -
    \left( \eta_x - f_{\tau,x} \right)
    \frac{\sigma'(f)_{\tau,x}}{\sigma(f)_{\tau,x}}
    \mathcal L_\infty(f)_{\tau,x} \\
  & = \left( \frac{g(\eta_x)}{\sigma\left( f_\tau\left(\frac{x}{N}\right)
    \right)} - \left( \eta_x - f_\tau\left(\frac{x}{N}\right) \right)
    \frac{\sigma'\left( f_\tau\left(\frac{x}{N}\right)
    \right)}{\sigma\left( f_\tau\left(\frac{x}{N}\right)
    \right)} \right)  \mathbf{L}_\infty (f_\tau)\left(\frac{x}{N}\right) +
    g(\eta_x) R_\tau ^N(x)
\end{align*}
with a remainder $R_\tau^N(x)$ estimated as follows:
\begin{align*}
  R_\tau^N(x)
  & := \mathtt A : \nabla_x ^2 [\sigma(f_\tau)]\left(\frac{x}{N}\right)
    - N^2 \sum_{y \in \T^d _N}  p(y-x) \left[ \sigma\left( f_\tau \left(
    \frac{y}{N} \right) \right) - \sigma\left( f_\tau \left(
    \frac{x}{N} \right) \right) \right] \\
  & = \mathtt A : \nabla_x ^2 [\sigma(f_\tau)]\left(\frac{x}{N}\right)
    - N^2 \sum_{y \in \T^d _N}  p(y-x) \Bigg[ \nabla
    \left[ \sigma\left( f_\tau \right) \right] \left(
    \frac{x}{N} \right)\frac{y-x}{N} \\
  & \hspace{5cm} + \frac1{2 N^2}
    \nabla^2 \left[ \sigma \left( f_\tau \right) \right] \left(
    \frac{x}{N} \right) : (y-x)^{\otimes 2} + \cO
    \left( \frac{\| \nabla f_\tau \|_{C^2}}{N^3} \right) \Bigg] \\
  & = \cO \left( \frac{\| \nabla f_\tau \|_{C^2}}{N} \right) \leq \cO \left( \frac{ e^{-C\tau}}{N} \right) 
\end{align*}
where we have used the identities
\begin{equation*}
  \sum_{y \in \T^d _N}  p(y-x) (y-x) = \gamma =0, \qquad
  \frac12 \sum_{y \in \T^d _N}  p(y-x) (y-x)^{\otimes 2} = \mathtt A.
\end{equation*}

The conservation of mass implies
\begin{equation*}
  \int_{\sX_N} \left( \sum_{x \in \T^d_N} A_{\tau,x} ^N \right)
  G^N_\tau \dd \vartheta_{f_\infty} ^N = \int_{\sX_N} \left( \sum_{x \in
      \T^d_N} A_{\tau,x} ^N \right) \dd \vartheta_{f_\tau} ^N =0,
\end{equation*}
thus we may remove a constant: we remove the average of $\Phi_{t-\tau} := e^{(t-\tau)\cL_N} \Phi$ with respect to the local equilibrium measure and replace $\Phi_{t-\tau}$ by 
\begin{equation*}
  \tilde \Phi_{t,\tau} := e^{(t-\tau)\cL_N} \Phi - \int_{\sX_N} 
  \left( e^{(t-\tau)\cL_N} \Phi \right) \dd \vartheta_{f_\tau}
  ^N = e^{(t-\tau)\cL_N} \left[ \Phi - \left( \int_{\sX_N} \Phi
      \dd \vartheta_{f_\tau}^N \right) \right].
\end{equation*}
This results in the estimate
\begin{align}
  \nonumber
  I_t ^N
  & = \int_0 ^t \int_{\sX_N} \Phi_{t,\tau}(\eta) \left(
    \sum_{x \in \T^d_N} A^N_{\tau,x} \right) \dd
    \vartheta_{f_\tau}^N \\
  \nonumber
  & = \int_0 ^t \int_{\sX_N} \tilde \Phi_{t,\tau}(\eta) \left(
    \sum_{x \in \T^d_N} A^N_{\tau,x} \right) \dd
    \vartheta_{f_\tau}^N \\
  \label{eq:I-inter1-zrp}
  & = \int_0 ^t \int_{\sX_N} \tilde \Phi_{t,\tau}(\eta) \left(
    \sum_{x \in \T^d_N} \tilde A^N_{\tau,x} \right) \dd
    \vartheta_{f_\tau} ^N + \cO\left(\frac{1}{N}\right)
\end{align}
where $\tilde A^N_{\tau,x}$ is defined as follows (we have added
terms independent of $\eta$ thanks to the zero average condition
on $\tilde \Phi_{t,\tau}$ with respect to the measure
${\rm d}\vartheta_{f_\tau}$):
\begin{equation*}
  \tilde A^N_{\tau,x} :=  \left\{ g(\eta_x) - \sigma\left(
      f_\tau\left(\frac{x}{N}\right) \right) - \sigma'\left(
      f_\tau\left(\frac{x}{N}\right) \right) \left[ \eta_x -
      f_\tau\left(\frac{x}{N}\right) \right] \right\}
  \frac{ \mathbf{L}_\infty (f)\left(\frac{x}{N}\right)}{\sigma\left(
      f_\tau\left(\frac{x}{N}\right) \right)}.
\end{equation*}

Note that $\tilde A^N_{\tau,x}$ now also has zero average against
${\rm d} \vartheta_{f_\tau} ^N$. We have used the estimate on
$R^N_\tau(x)$ and the Lipschitz bound
$[\Phi_{t,\tau}]_{\Lip(\sX_N)} \le 1$ (from the microscopic
stability) to get
\begin{align*}
  & \sum_{x \in \T^d_N} \int_0 ^t R^N_\tau (x) \left(
    \int_{\sX_N} \tilde \Phi_{t,\tau}(\eta) g(\eta_x)
    \dd \vartheta_{f_\tau} ^N \right) \dd \tau \\
  & \lesssim [\Phi_{t,\tau}]_{\Lip(\sX_N)}   
   \sum_{x \in \T^d_N} \int_0 ^t 
    \left(  \int_{\sX^2} R^N_\tau (x) \frac{|\eta_x-\zeta_x|}{N^d} g(\eta_x)    \dd n_{f_{\tau,x}} ^2(\eta_x,\zeta_x) \right) \lesssim \frac{1}{N}.
\end{align*}

\subsection{Step 2. Introduction of an intermediate scale}

We consider an intermediate scale $\ell < N$ so that $\ell$
divides $N$. This scale will be
chosen later in terms of $N$. We then form sub-sums over a
covering of $\T^d_N$ by non-overlapping cubes of size
$\ell \in \{1,\dots,N\}$. Let $\cR^\ell_N \subset \T^d_N$ be the
net of centres of these non-overlapping cubes denoted 
\begin{equation*}
  \cC^\ell _x := \{ y \in \T^d_N \ : \ \| x - y \|_\infty \le \ell
  \} \quad \text{for} \quad x \in \cR^\ell_N.
\end{equation*}
Then we rearrange~\eqref{eq:I-inter1-zrp} into
\begin{align*}
  I_t ^N
  & = \sum_{x \in \cR_N^\ell}  \int_0 ^t \int_{\sX_N} \tilde
    \Phi_{t,\tau}(\eta) \left( \sum_{x' \in \cC_x ^\ell} \tilde
    A^N_{\tau,x'} \right) \dd \vartheta_{f_\tau} ^N +
    \cO\left(\frac{1}{N}\right) \\
  & = (2\ell+1)^d \sum_{x \in \cR_N^\ell}  \int_0 ^t \int_{\sX_N}
    \tilde \Phi_{t,\tau}(\eta) \hat A^N_{\tau,x} \dd \vartheta_{f_\tau} ^N +
    \cO\left(\frac{1}{N}\right) + \cO \left(
    \frac{\ell }{N}
    \right)
\end{align*}
where $\hat A^N_{\tau,x}$ is defined as follows (all the sums contained in this term are normalised):
\begin{equation*}
  \hat A^N_{\tau,x} :=  \left\{ \langle g(\eta) \rangle_{\cC_x ^\ell} -
    \sigma\left( f_\tau\left(\frac{x}{N}\right) \right) -
    \sigma'\left( f_\tau\left(\frac{x}{N}\right) \right) \left[
      \langle \eta \rangle_{\cC_x ^\ell} - f_\tau\left(\frac{x}{N}\right)
    \right] \right\}
  \frac{ \mathbf{L}_\infty (f)\left(\frac{x}{N}\right)}{\sigma\left(
      f_\tau\left(\frac{x}{N}\right) \right)}
\end{equation*}
with the following notation for the local average
\begin{equation*}
  \langle F(\eta) \rangle_{\cC_x^\ell} := \frac{1}{(2 \ell + 1)^d}
  \sum_{x' \in \cC_x^\ell} F\left(\eta_{x'}\right) \quad \text{ for
  } F=F(\eta_x).
\end{equation*}

Note that $\hat A^N_{\tau,x}$ only depends on values of $\eta$ in
the box $\cC_x ^\ell$. 
From now on ${\bf P}_{\cC}$ denotes the marginal projection onto the
sites in $\cC$ with respect to  $\vartheta_{f_\tau}$:
\begin{align} \label{def: ell projection}
  {\bf P}_{\cC_x ^{\ell}} \tilde \Phi_{t,\tau} := \int_{\{\eta_y \in
  \sX, \; y \not \in \cC_x ^{\ell}\}} \tilde \Phi_{t,\tau}(\eta)
  \prod_{y \not \in \cC_x ^{\ell}} \dd n_{f_{\tau.y}}(\eta_y). 
\end{align}

We have to estimate a new error term when replacing the local
average of $\tilde A$ by $\hat A$. Since $\tilde \Phi_{t,\tau}$
has zero average, this error term reduces to
\begin{equation*}
  \cE_A = \langle \eta \rangle_{\cC_x^\ell} \sum_{y \in \cC_x^\ell} \left[
    \sigma\left( f_\tau\left(\frac{y}{N}\right) \right) -
    \sigma\left( f_\tau\left(\frac{x}{N}\right) \right)
  \right] = \cO \left( \frac{\ell}{N} e^{-C\tau} \right)
  \langle \eta \rangle_{\cC_x^\ell}
\end{equation*}
where the exponentially small in time term appears due to the decay of the Lipschitz norm of $f_\tau$. Then, we use the Lipschitz bound on $\Phi$ to get
\begin{equation}
  \label{eq:lip-average-phi}
  \left| {\bf P}_{\cC_x ^{\ell}}  \tilde{\Phi}_{t, \tau} \right| \leq N^{-d} \int_{\zeta \in \sX^{\T^d_\ell}} \sum_{z \in \cC_x ^{\ell}}|\eta_z - \zeta_z| \dd \vartheta_{f_\tau} ^{N,\cC_x ^\ell} 
\end{equation}
where $\zeta$ denotes the variables in $\cC_x^\ell$ used in the average subtracted from $\Phi$ in the definition of $\tilde \Phi$, and $\vartheta_{f_\tau} ^{N,\cC_x ^\ell}$ denotes the restriction of the product measure $\vartheta_{f_\tau} ^N$ to $\cC_x ^\ell$. This yields
\begin{align*}
  & \sum_{x \in \cR_N^\ell}  \int_0 ^t \int_{\sX_N} \tilde
    \Phi_{t,\tau}(\eta) \cE_A \dd \vartheta_{f_\tau} ^N \\
  & \lesssim \frac{\ell}{N^{d+1}} \sum_{x \in \cR_N^\ell}  \int_0 ^t \int_{\sX_\ell^2} \left( \sum_{z \in \cC_x ^{\ell}}|\eta_z^{(\ell)} - \zeta_z ^{(\ell)} | \right)      \langle \eta \rangle_{\cC_x} ^\ell \dd \left( \vartheta_{f_\tau} ^{N,\cC_x ^\ell}    \right)^{\otimes 2}\left(\eta^{(\ell)},\zeta^{(\ell)}\right) \dd \tau \lesssim \frac{\ell }{N}e^{-Ct}
\end{align*}
and $\eta^{(\ell)}$ denotes a blind variable in $\sX_\ell$ and
$\vartheta_{f_\tau} ^{N,\cC_x ^\ell}$ denotes the restriction of
the product measure $\vartheta_{f_\tau} ^N$ to $\cC_x ^\ell$. Note
that the average of $\hat A^N _{\tau,x}$ against
${\rm d} \vartheta_{f_\tau} ^N$ is not zero anymore but similar
estimates can be used to prove it is small:
\begin{equation}
  \label{eq:averagehatA}
  \int_{\sX_N} \hat A^N_{\tau,x} \dd\vartheta_{f_\tau} ^N =
  \cO \left( \frac{\ell}{N} e^{-C\tau} \right).
\end{equation}

We then ``flatten'' the local equilibrium measure locally in each $\ell$-cube with error
\begin{align} \label{eq:decomp-consis}
  \nonumber
  I_t ^N
  & = (2\ell+1)^d \sum_{x \in \cR_N^\ell}  \int_0 ^t \int_{\sX_N}
    \tilde \Phi_{t,\tau}(\eta) \hat A^N_{\tau,x} \dd
    \vartheta_{f_\tau} ^N \\
    & = 
    (2\ell+1)^d \sum_{x \in \cR_N^\ell}  \int_0 ^t \int_{\sX_\ell}
    \tilde \Phi_{t,\tau}(\eta) \hat A^N_{\tau,x} \dd
    \vartheta_{f_{\tau,x}} ^\ell + \cO \left(
    \frac{\ell^{\frac{3d}{2}+1}}{N} \ln \left( \frac{N }{\ell^{d+1}} \right)       \right)
\end{align}
where $\vartheta_{f_{\tau,x}} ^\ell$ is the constant equilibrium measure in the $\ell$-box with fixed parameter $f_\tau(x/N)$ at every point of the box. In the second line of \eqref{eq:decomp-consis} we have used the regularity of $f_\tau$ and the finite exponential moments. The proof of this error is not complicated but quite technical, let us give more details. Indeed we have 
\begin{align*} 
  & \sum_{x \in \cR_N^\ell}\int_0 ^t \int_{\sX_N} \left({\bf P}_{\cC_x ^{\ell}} \tilde \Phi_{t,\tau} \right)  \hat A^N_{\tau,x} \left[  \dd \vartheta_{f_\tau}
  ^{N,\cC^{\ell}_x} -  \dd  \vartheta_{f_{\tau,x}} ^{\ell}  \right] \dd \tau \\
  &\lesssim \frac{\ell^d}{N^d} \sum_{x \in \cR_N^\ell} \int_0 ^t \int_{\sX_N}
    \left\vert \hat A^N_{\tau,x} \right\vert \langle \eta \rangle_{\cC_x ^\ell}  \left\vert  \dd \vartheta_{f_\tau} ^{N,\cC^{\ell}_x} -  \dd  \vartheta_{f_{\tau,x}} ^{\ell}  \right\vert 
    \dd \tau \\
  & \lesssim  \int_0 ^t \| \hat{A}^N_{\tau,x} \|_{L^2(\dd  \vartheta_{f_{\tau,x}}
  ^{\ell} )} \left[  \frac{\ell^{d+1} }{N} \ln \left( \frac{N }{\ell^{d+1}} \right) \right] e^{-C\tau}  \dd \tau
  \\ & \lesssim \ell^{-d/2} \left[  \frac{\ell^{d+1} }{N} \ln \left( \frac{N }{\ell^{d+1}} \right) \right] = \frac{\ell^{d/2+1}}{N} \ln \left( \frac{N }{\ell^{d+1}} \right).
\end{align*}
In the first line we have used the Lipschitz bound on $\Phi$ and moments on the Gibbs measure to estimate $|{\bf P}_{\cC_x ^{\ell}} \tilde \Phi_{t,\tau}| \lesssim \ell^d N^{-d} \langle \eta \rangle_{\cC^\ell_x}$. In the last line we have used the estimate $\| \hat{A}^N_{\tau,x}\|_{L^2(\dd \vartheta^\ell_{f_{\tau, x}})} \lesssim e^{-C\tau} \ell^{-d/2}$ following from the product structure of the Gibbs measure, together with $\mathbf{L}_\infty (f_t)$ decreasing exponentially to zero as $t \to \infty$. In the second line, we have estimated the difference between the two measures in $L^2 ((\vartheta^\ell_{f_{\tau,x}})^{-1})$. We use
\begin{align*} 
  & \left\vert \frac{\dd \vartheta_{f_\tau} ^{N,\cC^{\ell}_x} }{\dd  \vartheta_{f_{\tau,x}} ^{\ell} } -1 \right\vert  = \left\vert \frac{\prod_{z \in \mathcal{C}_x^\ell}  \sigma ( f_{\tau,z} )^{\eta_z} Z( \sigma ( f_{\tau,x}  ))}{ \prod_{z \in \mathcal{C}_x^\ell} \sigma ( f_{\tau,x} )^{\eta_z}  Z( \sigma ( f_{\tau,z}))} -1 \right\vert \\
  & = \exp \left( \sum_{z \in \mathcal{C}_x^\ell} \eta_z \ln \left[ \frac{\sigma (f_{\tau,z})-\sigma (f_{\tau,x})}{\sigma (f_{\tau,x})} +1 \right] + \sum_{z \in \mathcal{C}_x^\ell} \ln \left[ \frac{Z(\sigma (f_{\tau,x}))}{Z(\sigma(f_{\tau,z}))}  \right]\right) -1 \\ 
  & \le \exp \left( \sum_{z \in \mathcal{C}_x^\ell} \eta_z \ln \left[ C'e^{-Ct} \frac{\ell}{N} + 1 \right] + \sum_{z \in \mathcal{C}_x^\ell} \ln \left[ \frac{Z( \sigma (f_{\tau,x}))}{Z(\sigma(f_{\tau,z}))}  \right]\right) -1 \\
  & \le \exp \left( C'e^{-Ct} \frac{\ell}{N} \sum_{z \in \mathcal{C}_x^\ell} \eta_z  + \sum_{z \in \mathcal{C}_x^\ell} \left\vert f_{\tau,z} - f_{\tau,x} \right\vert \right) -1 \le \exp \left( C' e^{-Ct}  \frac{\ell^{d+1}}{N} \left[ \langle \eta \rangle_{\cC_x ^\ell} +1 \right] \right) -1, 
\end{align*} 
where we used that the partition function $Z$ satisfies by construction
\begin{equation*}
  \ln \left(\frac{Z( \sigma( f_\tau(x/N)))}{Z( \sigma ( f_\tau(z/N)))} \right)  \le \left| \int_{\sigma(f_\tau(z/N))}^{\sigma(f_\tau(x/N))} \frac{\sigma^{-1}(h)}{h} \dd h \right| \lesssim \left\vert f_\tau(z/N) - f_\tau(x/N) \right\vert,
\end{equation*}
by using the lower bounds on $f_\tau$ and the regularity of $\sigma$. We deduce 
\begin{equation}\label{eq:flatten measure}
  \begin{split}  
    & \left( \int_{\sX_N} \left\vert \frac{\dd \vartheta_{f_\tau}  ^{N,\cC^{\ell}_x} }{\dd  \vartheta_{f_{\tau,x}} ^{\ell} } -1 \right\vert^2 \dd  \vartheta_{f_{\tau,x}}^\ell \right)^{1/2} \lesssim \left( \int_{\sX_N}  \left\vert \exp \left( C' \frac{\ell^{d+1}}{N} \left[ \langle \eta \rangle_{\cC_x ^\ell} +1 \right]  \right) -1\right\vert^2 \dd  \vartheta_{f_{\tau,x}}^\ell \right)^{1/2}\\
    & \lesssim \left( \int_{\sX_N} 1_{ \{ \langle \eta \rangle_{\cC_x^\ell} \leq M \} }  \left\vert \exp \left( C' \frac{\ell^{d+1}}{N} (M+1) \right) -1 \right\vert^2 \dd  \vartheta_{f_{\tau,x}}^\ell \right)^{\frac12} \\
    & \hspace{3cm} + \left( \int_{\sX_N} 1_{\{ \langle \eta \rangle_{\cC_x^\ell} >M \} }  \left\vert \exp \left( C' \frac{\ell^{d+1}}{N} \left[ \langle \eta \rangle_{\cC_x ^\ell} +1 \right] \right) -1 \right\vert^2 \dd  \vartheta_{f_{\tau,x}}^\ell \right)^{\frac12}
  \end{split}
\end{equation}
for some constant $C'>0$, with a cut-off $M = M(N)$ to be chosen later in terms of $N$. Assuming $\ell^{d+1} \le N$, the first term is of order $ \mathcal{O}\left( M \ell^{d+1} N^{-1} \right)$, while for the second term we use the exponential moment relation $\mathbb{E}_{ \vartheta_{f_{\tau,x}}^N } [ e^{\theta \eta_0} ] = \frac{Z( \sigma(f_{\tau,x}) e^{\theta})}{Z(\sigma(f_{\tau,x}))}$ for any $\theta \in \R$ (valid and finite under our assumptions on the jump rate $g$) to deduce
\begin{equation*}
  \mathbb{E}_{ \vartheta_{f_{\tau,x}}^N } \left[ e^{\theta \langle \eta \rangle_{\mathcal{C}_x^\ell} } \right] = \prod_{z \in \mathcal{C}_x^\ell}  \mathbb{E}_{  \vartheta_{f_{\tau,x}}^N } [ e^{\theta \ell^{-d}  \eta_z } ] =  \prod_{z \in \mathcal{C}_x^\ell} \frac{Z( \sigma(f_{\tau,z}) e^{\theta \ell^{-d}})}{Z(\sigma(f_{\tau,z}))}, 
\end{equation*}
and since 
\begin{equation*}
  \left| \ln \left(\frac{Z( \sigma( f_{\tau,z}))}{Z( \sigma(f_{\tau,z}) e^{\theta \ell^{-d}})} \right) \right| = \left| \int_{\sigma(f_{\tau,z}) e^{\theta \ell^{-d}} }^{\sigma(f_{\tau,z})} \frac{\sigma^{-1}(r)}{r} \dd r \right| \le C \left\vert 1-e^{\theta \ell^{-d}}  \right\vert, 
\end{equation*}
we obtain indeed, for a fixed $\theta$ and $\ell$ large enough so that $\theta\ell^{-d} \le 1$:
\begin{equation*}
 \mathbb{E}_{ \vartheta_{f_{\tau,x}}^N } \left[ e^{\theta \langle \eta \rangle_{\mathcal{C}_x^\ell} } \right] 
 \le \prod_{z \in \mathcal{C}_x^\ell} e^{C \left\vert 1-e^{\theta \ell^{-d}}  \right\vert} = \exp\left( C \sum_{z \in \mathcal{C}_x^\ell}\left\vert 1-e^{\theta \ell^{-d}}  \right\vert \right) \lesssim \exp\left( C \theta \right) <\infty.
\end{equation*}
Exploiting this finite exponential moment for the $\ell$-averages, and applying H\"{o}lder's inequality for some $\theta$ with conjugate $\theta'$, we estimate the second term in~\eqref{eq:flatten measure}: 
\begin{align*} 
  C'\left(\int_{\sX_N} 1_{\{ \langle \eta \rangle_{\cC_x^\ell} >M \} } \dd  \vartheta_{f_{\tau,x}}^\ell \right)^{\frac{1}{\theta'}} \leq C''  \left( \int_{\sX_N} 1_{\{ \langle \eta \rangle_{\cC_x^\ell} >M \}} e^{ - 2 \langle \eta \rangle_{\cC_x^\ell} } \dd  \vartheta_{f_{\tau,x}}^\ell \right)^{\frac{1}{2\theta'}} \leq C'' e^{-\frac{M}{\theta'}}
\end{align*} 
for some explicit constants $C'' > 0 $ depending on the limit solution and on the exponential moment bounds, but not on the scaling parameters $\ell, N$. Inserting this into \eqref{eq:flatten measure} we get 
\begin{equation}
  \label{eq:flatten measure2}
  \left( \int_{\sX_N} \left\vert \frac{\dd \vartheta_{f_\tau}          ^{N,\cC^{\ell}_x} }{\dd  \vartheta_{f_{\tau,x}}^{\ell} } -1       \right\vert^2 \dd  \vartheta_{f_{\tau,x}}^\ell \right)^{1/2} \lesssim M \frac{\ell^{d+1}}{N} + e^{-\frac{M}{2\theta'}}.
\end{equation} 
We finally choose $M:=2 \theta'\ln(N \ell^{-d-1})$, which shows that the error of flattening the local Gibbs measure measure is indeed of order $\ell^{d+1} N^{-1} \ln(N \ell^{-d-1})$. 

\subsection{Step 3. Projection on the local equilibrium}

Next, we remove and add the local equilibrium, i.e. the average
over all configurations with the same mass in the $\ell$-box:
\begin{align}
  \nonumber
  & \sum_{x \in \cR_N^{\ell}}  \int_0 ^t \int_{\sX_\ell}
   {\bf P}_{\cC_x ^{\ell}}  \tilde \Phi_{t,\tau} (\eta) \hat A^N_{\tau,x} \dd
    \vartheta_{f_{\tau,x}} ^\ell \\
  \nonumber
  & = \sum_{x \in \cR_N^{\ell}}  \int_0 ^t \int_{\sX_\ell} \left[
   {\bf P}_{\cC_x ^{\ell}}  \tilde \Phi_{t,\tau}(\eta) - \Pi_x ^\ell \ {\bf P}_{\cC_x ^{\ell}} \tilde \Phi_{t,\tau}(\eta) 
     \right] \hat
    A^N_{\tau,x} \dd \vartheta_{f_{\tau,x}}
    ^\ell \\
  \nonumber
  & \hspace{2.5cm} + \sum_{x \in
    \cR_N^{\ell}}  \int_0 ^t \int_{\sX_\ell}  \Pi_x ^\ell {\bf P}_{\cC_x ^{\ell}} \tilde \Phi_{t,\tau} (\eta)
      \hat A^N_{\tau,x}
    \dd \vartheta_{f_{\tau,x}} ^\ell =: (2\ell+1)^{-d} J_t^{N} + (2\ell+1)^{-d} \tilde J_t^{N}
\end{align}
where $\Pi_x ^{\ell} \varphi$ projects any function
$\varphi(\eta^{(\ell)})$ of $\eta^{(\ell)} \in \sX_\ell$ onto the
hyperplan
\begin{equation*}
  \Omega_m := \left\{ \eta^{(\ell)} \in \sX_\ell \, : \, \langle
    \eta^{(\ell)} \rangle_{\cC_z} = m \right\}
\end{equation*}
of constant mass $m$. The precise formula is
\begin{equation}
  \label{eq:pi-average}
  \Pi_x ^\ell \varphi(\eta^{(\ell)}) = [\Pi_x ^\ell \varphi](\langle
  \eta^{(\ell)} \rangle_{\cC_x ^\ell}) = \int_{\Omega_{\langle \eta^{(\ell)}
      \rangle_{\cC_x ^\ell}}} \varphi(\eta^{(\ell)})
  \frac{\dd \vartheta_{f_{\tau,x}}^\ell
    (\eta^{(\ell)})}{\vartheta_{f_{\tau,x}}^\ell(\Omega_{\langle
      \eta^{(\ell)} \rangle})}.
\end{equation}
Combined with~\eqref{eq:decomp-consis} we have therefore the decomposition
\begin{align}
  \label{eq:decomp-consis2}
  I_t ^N = J^N_t + \tilde J^N_t + \cO \left(
    \frac{\ell^{\frac{3d}{2}+1}}{N} \ln \left( \frac{N }{\ell^{d+1}} \right)
      \right).
\end{align}

\subsection{Step 4. Local coercivity estimate}

To estimate the first term $J^N_t$ we use the
Cauchy-Schwarz inequality and the Poincaré inequality in the cube $\cC_x^\ell$, see \cite{MR1415232}. 
(The constant in
the Poincaré inequality is independent of the number of particles
in this cube and proportional to $\ell^2$.) 
\begin{equation}
\begin{split}
&J_t^{N} \leq (2\ell+1)^d \sum_{x \in \cR_N^{\ell}}  \int_0 ^t \int_{\sX_\ell} \left[
   {\bf P}_{\cC_x ^{\ell}} \tilde \Phi_{t,\tau} (\eta) - \Pi_x ^\ell \ {\bf P}_{\cC_x ^{\ell}} \tilde \Phi_{t,\tau} (\eta) 
     \right] \hat
    A^N_{\tau,x} \dd \vartheta_{f_{\tau,x}}^\ell \\ 
    &\lesssim
   \ell^{d} \sum_{x \in \cR_N^{\ell}} \int_0 ^t  \left[ \int_{\sX_\ell} \left[
    {\bf P}_{\cC_x ^{\ell}} \tilde \Phi_{t,\tau} (\eta) 
    - \Pi_x ^\ell \ {\bf P}_{\cC_x ^{\ell}} \tilde \Phi_{t,\tau} (\eta) 
     \right] ^2  \dd \vartheta_{f_{\tau,x}}^\ell  \right]^{1/2} \|  A^N_{\tau,x} \|_{L^2_{\vartheta^\ell_ {f_{\tau,x}}}}
        \dd \tau \\
        &\lesssim \ell^{d/2} \sum_{x \in \cR_N^{\ell}} \int_0 ^t e^{-C\tau}
        \left[ \int_{\sX_\ell} \left[
   {\bf P}_{\cC_x ^{\ell}}  \tilde  \Phi_{t,\tau} (\eta) - \Pi_x ^\ell \  {\bf P}_{\cC_x ^{\ell}} \tilde  \Phi_{t,\tau} (\eta) 
     \right] ^2  \dd \vartheta_{f_{\tau,x}}^\ell  \right]^{1/2}  \dd \tau, 
    \end{split}
\end{equation}
where in the second line we have used again the estimate $\| A^N_{\tau,x} \|_{L^2} \lesssim e^{-C\tau} \ell^{-d/2}$. We then apply
Poincaré inequality to get the upper bound
\begin{equation}
  J_t^{N} \leq \ell^{1+ d/2} \sum_{x \in \cR_N^{\ell}} \int_0 ^t  e^{-C\tau} \sqrt{\mathcal D_{\vartheta^\ell _{f_{\tau,x}}} \left( {\bf P}_{\cC_x ^{\ell}} \tilde \Phi_{t,\tau} \right)} \dd \tau.
\end{equation}
Here the Dirichet form $\mathcal D_{\vartheta^\ell _{f_{\tau,x}}}(\varphi)$ is the $L^2$ entropy production functional on the $\ell$-box with respect to the reference equilibrium measure $\vartheta^\ell _{f_{\tau,x}}$:
\begin{equation*}
  \mathcal D_{\vartheta^\ell _{f_{\tau,x}}}(\varphi) := \sum_{y_1,y_2 \in \cC_x^\ell} \int_{\sX_\ell} p(y_1-y_2) g(\eta_{y_1}) \left[ \varphi^{y_1 y_2} - \varphi  \right]^2 \dd \vartheta ^\ell _{f_{\tau,x}}.
\end{equation*}
The Cauchy-Schwarz inequality then implies
\begin{align} 
  J_t^{N} 
  & \leq \ell^{1+ d/2} \sum_{x \in \cR_N^{\ell}} \int_0 ^t 
    e^{-C\tau} \sqrt{ \mathcal D_{\vartheta^\ell _{f_{\tau,x}}} \left(\tilde \Phi_{t,\tau} \right)} \dd \tau \lesssim  \int_0 ^t \ell^{1+d/2} \frac{N^{d/2}}{\ell^{d/2}} e^{-C\tau} \sqrt{ \sum_{x \in \cR_N^{\ell}} \mathcal D_{\vartheta^\ell _{f_{\tau,x}}} \left( \tilde \Phi_{t,\tau} \right)} \dd \tau \nonumber \\
  & \lesssim \ell N^{d/2} \int_0 ^t e^{-C\tau} \left[ \sum_{y_1 \sim y_2 \in \cR_N^{\ell}} \int_{\sX_N} p(y_1-y_2) g(\eta_{y_1}) \left[ \tilde \Phi_{t,\tau} ^{y_1y_2}- \tilde \Phi_{t,\tau} \right]^2 \dd \vartheta^N  _{f_{\tau,x}} \right]^{1/2} \dd \tau  \nonumber\\
  & \lesssim \ell N^{d/2} \int_0 ^t e^{-C\tau} \sqrt{\mathcal D_{\vartheta^\ell _{f_{\tau,x}}} \left(\tilde \Phi_{t,\tau} \right)} \dd \tau.
    \label{eq:dirichlet form to estimate}
\end{align}

We now aim to establish sufficiently strong decay in $N$ on the Dirichlet form in order to obtain the decay for the term $J_t^N$. This estimate is in fact a control of regularity directly at the discrete level of the many-particle system. We propose two methods, which both have interest per se. The first one relies on Lemma~\ref{lem:jumpdistance estim} to prove that a ``jump distance'' that we introduce is non-increasing, but it is restricted to dimension $d=1$ and test functions $\Phi$ that are simple modulated sums. The second method is based on establishing a lower bound on the minimum number of jumps within a given time interval up to an exponentially small error of large deviation in $N$, thanks to the non-degeneracy condition on the jump rate.

We wish to estimate
\begin{equation*}
  \left\vert \Phi_t (\eta^{zz'}) - \Phi_t (\eta)\right\vert  = \left\vert \mathbb{E} \Big[\Phi_0 ((\eta^{zz'})^t) - \Phi_0 (\eta^t) \Big] \right\vert,
\end{equation*}
which appears in the Dirichlet energy. Employing Lemma \ref{lem:jumpdistance estim}, we note that $\operatorname{supp}(\pi_t^N) \subset \{ \| \eta - \zeta\|_{\ell^1} \leq 2 \}$ for all times: given that the mass is conserved and $\zeta^0=(\eta^0)^{zz'}$ and $\eta^0$ have the same mass, so do $\eta$ and $\zeta$ for all times in the support of $\pi_t$, and the above support condition implies that all pairs of configuration $(\eta,\zeta)$ in the support of $\pi_t$ satisfy $\zeta=\eta$ or $\zeta^t=(\eta^t)^{x_-x_+}$ for two sites $x_-, x_+ \in \T_N$ not necessarily neighbours (it is a \emph{generalised jump} so to speak). When the two points $x_-=x_+$ become equal, they remain so, thanks again to weak contraction of the support in Lemma~\ref{lem:jumpdistance estim}. 
\medskip

\noindent
\textit{Method 1: Monotonicity of the jump distance in dimension $1$.}
Assume $d=1$, $p$ to be the symmetric nearest-neighbour transition function, and consider $\Phi_0 = N^{-1}\sum_{x \in \T_N}\eta_x \phi_x$ with $\phi \in C^1(\T)$: such particular form of the initial test function is an assumption of ``discrete smoothness'' at the initial time, i.e. the projections of $\Phi_0$ have similar laws for nearby sites. These assumptions are restrictive, but  the stability estimate we prove is a strong stability on the level of the trajectories. Given $\eta,\zeta \in \sX_N$ with same mass $\sum_{x \in \T_N} \eta_x=\sum_{x \in \T_N} \zeta_x$, we define the \textit{jump distance} $J(\eta, \zeta)$ as the minimum number of jumps to transform the configuration $\eta$ into the configuration $\zeta$. For given two neighbouring points $z \sim z'$ on the lattice,  $J(\eta, \eta^{zz'}) = 1$, i.e. the configurations $\eta$ an $\eta^{zz'}$ are one jump apart. Given the initial coupling measure $\pi_0^N(\eta, \zeta):=\delta_{\eta=\eta}\otimes \delta_{\zeta=\eta^{zz'}}$ on $\sX_N^2$, we then use the standard coupling described above to define the coupling measure $\pi_t$ of $(\eta_t,\zeta_t)$ at later times. And we have the formula
\begin{align*}
  \left\vert \Phi_t (\eta^{zz'}) - \Phi_t (\eta)\right\vert
  & = \left\vert \mathbb{E} \Big[\Phi_0 \left((\eta^{zz'})^t\right) - \Phi_0 \left(\eta^t\right) \Big] \right| 
   = \left| \int_{\eta,\zeta \in \sX_N} \Big[\Phi_0(\zeta) - \Phi_0(\eta) \Big] \dd \pi_t(\eta,\zeta) \right| \\
  & = \left| \int_{\eta,\zeta \in \sX_N} N^{-1} \left[ \sum_{x \in \T_N} \phi_x \left( \eta _x - \zeta_x \right) \right] \dd \pi_t(\eta,\zeta) \right| \\
  & \lesssim N^{-2} \int_{\eta,\zeta \in \sX_N} J\left(\eta,\zeta\right) \dd \pi_t(\eta,\zeta)
\end{align*}
where in the last line we have used that $\eta$ can be connected to $\zeta$ by $J(\eta,\zeta)$ jumps between two neighbouring sites, and each of them costs an error $N^{-1}$ thanks to the $C^1$ regularity of $\phi$. Indeed for each jump $\zeta=\eta^{zz'}$ with $z \sim z'$:
\begin{equation*}
  \sum_{x \in \T_N} \phi_x \left( \eta - \eta^{zz'} \right) = \phi_{z'} - \phi_z = \mathcal O(N^{-1}).
\end{equation*}
We finally claim that
\begin{equation}
  \label{eq:claim-J}
  \int_{\eta,\zeta \in \sX_N} J\left(\eta,\zeta\right) \dd \pi_t(\eta,\zeta) \le \int_{\eta,\zeta \in \sX_N} J\left(\eta,\zeta\right) \dd \pi_0(\eta,\zeta) =2.
\end{equation}
In order to prove \eqref{eq:claim-J}, we use the remark above that 
$\zeta^t=(\eta^t)^{x_-x_+}$ for two sites $x_-, x_+ \in \T_N$ not necessarily neighbours, and in fact it follows from our definition of the jump distance that these two sites are exactly at distance $J(\eta,\zeta)$ on the lattice (where the lattice distance is the minimal number of edges to connect two points). We then differentiate in time
\begin{align*} 
    & \frac{\dd}{\dd t} \int_{\sX_N^2} J(\eta, \zeta) \dd \pi_t^N(\eta, \zeta)
    = \int_{\sX_N^2} \widetilde{\mathcal{L}}_N J(\eta, \zeta)\ \dd \pi_t^N(\eta, \zeta) \\
    & = N^2 \int_{\sX_N^2} \sum_{x\sim x'} p(x'-x) \left( g(\eta_x)\wedge g(\zeta_x) \right) \left[J(\eta^{xx'}, \zeta^{xx'}) - J(\eta, \zeta)\right] \dd \pi_t^N(\eta, \zeta) \\
    & \quad + N^2 \int_{\sX_N^2} \sum_{x\sim x'}p(x'-x) \left(g(\eta_x)- g(\zeta_x)\right)_+ \left[J(\eta^{xx'}, \zeta ) - J(\eta, \zeta)\right] \dd \pi_t^N(\eta, \zeta) \\
    & \quad + N^2 \int_{\sX_N^2} \sum_{x\sim x'}p(x'-x) \left(g(\zeta_x)- g(\eta_x) \right)_+ \left[J(\eta, \zeta^{xx'}) - J(\eta, \zeta)\right] \dd \pi_t^N(\eta, \zeta). 
\end{align*}
The first term of the right hand side is zero. The second term of the right hand side is nonzero only at sites $x$ where $\eta_x > \zeta_x$, i.e. at $x=x_-$ and the third term of the right hand side is nonzero only at sites $x$ where $\zeta_x > \eta_x$ i.e. at $x=x_+$. Therefore
\begin{align*}
  & \frac{\dd}{\dd t} \int_{\sX_N^2} J(\eta, \zeta) \dd \pi_t^N(\eta, \zeta) \\
  &  = \frac{N^2}{2} \int_{\sX_N^2} \left(g(\eta_{x_-})- g(\eta_{x_-}-1)\right)_+ \left[J\left(\eta^{x_-(x_--1)}, \zeta\right) + J\left(\eta^{x_-(x_-+1)},\zeta\right) - 2 J(\eta, \zeta)\right] \dd \pi_t^N(\eta, \zeta) \\
  & + \frac{N^2}{2} \int_{\sX_N^2} \left(g(\zeta_{x_+})- g(\zeta_{x_+}-1)\right)_+ \left[J\left(\eta, \zeta^{x_+(x_+-1)} \right) + J\left(\eta,\zeta^{x_+(x_++1)}\right) - 2 J(\eta, \zeta)\right] \dd \pi_t^N(\eta, \zeta).
\end{align*}
Note finally that either in dimension $d=1$ we have
\begin{equation*}
  J\left(\eta^{x_-(x_--1)}, \zeta\right) + J\left(\eta^{x_-(x_-+1)},\zeta\right) =
  J\left(\eta, \zeta^{x_+(x_+-1)} \right) + J\left(\eta,\zeta^{x_+(x_++1)}\right) = 2 J(\eta,\zeta)
\end{equation*}
since one direction decreases the jump distance by $1$ while the other direction increases it by $1$. We conclude that $|\Phi_t (\eta^{zz'}) - \Phi_t (\eta)| \lesssim N^{-2}$,  which implies
\begin{align}
  \label{eq: J_t^N error}
  J_t^N \lesssim \ell N^{d/2} \int_0^t e^{-C\tau} \sqrt{\mathcal{D}_{\vartheta_{f_{\tau,x}}} \left(\Phi_{t,\tau}\right)} \dd \tau \lesssim \mathcal{O}\left( \ell N^{-1}\right).
\end{align}
\smallskip

\noindent
\textit{Method 2: The random walk viewpoint.} We now consider the random points $x_-(t)$ and $x_+(t))$ as jump processes following symmetric random walks in a given random environment, with respective jump rates $(g(\eta_{x_-})- g(\eta_{x_-}-1))_+$ and $\left(g(\eta_{x_+}+1)- g(\eta_{x_+})\right)_+$, provided we consider $\eta$ as given. Then $S(t) = x_+(t)-x_-(t)$ also defines a symmetric random walk in a random environment. Note that $S$ stop when it hits zero, i.e. $x_+=x_-$. The idea is to use quantitatively the recurrence of random walks in dimension $1$ and $2$.

The first step consists in proving a lower bound on the number of jumps per unit of time for the random walk $S(t)$ as long as it has not hit zero (where it then stays forever). We use the assumptions made on $g$, in particular $0 < g_1^* \leq |g(k) - g(k+1)| \leq g_2^* < \infty$. We claim that there is $\beta \in (0,1)$ small enough, and constants $C_1,C_2>0$, all depending on $g_1^*$ and $ g_2^*$, so that 
\begin{equation}
  \label{eq:claim-jumps}
  \mathbb{P}\left(\big[\text{number of jumps up to time } t \big] \leq \beta tN^2\right) \leq C_1 e^{- C_2 tN^2}.
\end{equation}

Let us prove this claim~\eqref{eq:claim-jumps}. Let $(T_k)_{k\ge 0}$ be the random waiting times for the jumps of $S(t)$, which follow exponential distribution with rate $\lambda_k$, i.e. $T_k \sim \lambda_k e^{-\lambda_k t}1_{t\geq 0}$ with density function $h_k$. We define $\omega_n := \sum_{k=1}^{n} T_k$ the time of the $n$-th jump. Having less than $\beta tN^2$ jumps before time $t$ means that the time of the $\beta tN^2$-th jump is greater $t$, so the error we want to estimate is
\begin{equation}
  \label{eq1_number}
  \mathbb{P}( \omega_{\beta tN^2} > t) = \int_t^\infty (  h_1 \ast h_2 \ast \cdots \ast h_{\beta tN^2} ) (s) \dd s.
\end{equation} 
Since, after rescaling time by $N^2$, $N^2 \lambda_i \in [N^2g_1^*, N^2g_2^*]$, for $i=1,\dots,\beta tN^2$, the $ \mathbb{P}( \omega_{\beta tN^2} > t)$  is upper bounded by  
\begin{equation*}
  \int_t^\infty ( \underbrace{ \bar{h} \ast \bar{h} \ast \cdots \ast \bar{h} }_{\beta tN^2 \text{ times}} ) (s) \dd s
\end{equation*}
where $ \bar{h} = N^2 ( \max \lambda_i) e^{-N^2 (\min \lambda_i) t}1_{t\geq 0} =N^2 g_2^* e^{-N^2 g_1^* t}1_{t\geq 0}  $. We compute inductively
$$ (\bar{h})^{ \ast \beta tN^2} (s) = (N^2g_2^*)^{\beta tN^2} \frac{s^{\beta tN^2-1} }{(\beta tN^2-1)!} e^{-N^2g_1^*s}.$$
Inserting this into \eqref{eq1_number}, we have 
\begin{align}
  \label{eq: No of jumps_est1}
  \mathbb{P}( \omega_{\beta tN^2} > t)
  &\leq \int_t^\infty 
    (N^2g_2^*)^{\beta tN^2} \frac{s^{\beta tN^2-1} }{(\beta tN^2-1)!} e^{-N^2g_1^*s}
    \dd s \nonumber \\
  & = \left( \frac{g_2^*}{g_1^*} \right)^{\beta tN^2}
    \left[  (g_1^*N^2)^{\beta tN^2}\int_t^\infty 
    \frac{s^{\beta tN^2-1} }{(\beta tN^2-1)!} e^{-N^2g_1^*s}
    \dd s\right] \\
  & = \left( \frac{g_2^*}{g_1^*} \right)^{\beta tN^2} \frac{1}{(\beta tN^2-1)!} \int_{g_1^*tN^2}^\infty y^{\beta tN^2-1} e^{-y} \dd y 
\end{align}
by change of variable. Then given $L$ real and $K$ integer, both large, we use
\begin{align*}
  \frac{1}{K!} \int_L ^{+\infty} y^K e^{-y} \dd y
  & = \frac{1}{K!} \left| \partial_{\lambda}^K \left( \int_L ^{+\infty} e^{-\lambda y} \dd y \right)_{|\lambda=1} \right| = \frac{1}{K!} \left| \partial_{\lambda}^K \left( \frac{e^{-\lambda L}}{\lambda} \right)_{|\lambda=1} \right| \\
  & \le \frac{1}{K!} \sum_{k=0} ^K \left( \begin{matrix} K \\ k \end{matrix} \right) L^k e^{-L} \le \frac{1}{K!} (L+1)^K e^{-L}.
\end{align*}
Applying it with $L:=g_1^* t N^2$ and $K:=\beta t N^2-1$ gives
\begin{align*}
  \mathbb{P}( \omega_{\beta tN^2} > t) \lesssim \left( \frac{g_2^*}{g_1^*} \right)^{\beta tN^2} \frac{(g_1^* t N^2)^{\beta t N^2 -1}}{( \beta tN^2-1)!} e^{-g_1^* t N^2}
\end{align*}
which gives the claimed estimate by the Stirling formula.

We then decompose $J_t ^N = J_t^{\text{NR}} + J_t^{\text{slow}} + J_t^{\text{R}}$ into three events: (i) when the number of jumps before time $t$ is at least $\beta tN^2$ and the random walk $S_t$ has not returned to zero within this time interval; (ii) when the number of jumps before time $t$ is less than $\beta tN^2$; and (iii) when the random walk $S_t$ has returned to zero within this time. The last case (iii) is simple since in this region $|\Phi_0(\zeta^t)-\Phi_0(\eta^t)|=0$  because the two configurations are equal at time $t$ (and remain so for all subsequent times). The second case (ii) is small thanks to the estimate~\eqref{eq:claim-jumps}, which gives an exponentially small error in $N$. We finally consider the first case (i), when the random walk $S$ has jumped ``a lot'' but still not returned to the origin:
\begin{align*}
 & J_t^{\text{NR}} \lesssim \ell N^{d/2} \int_0^t  e^{-C\tau} \dd\tau  \\ 
  &  \sqrt{ \sum_{y_1 \sim y_2} \int_{\sX_N} g(\eta_{y_1}) \left| \mathbb{E} \left( 1_{x_+(t) \not = x_-(t)} 1_{\# \text{jumps} \ge \beta tN^2} \left[ \Phi_0 \left(\eta_{t-\tau}\right) - \Phi_0 \left(\eta_{t-\tau}^{x_-(t-\tau) x_+(t-\tau)} \right) \right] \right) \right|^2 \dd \vartheta_{f_{\tau, x}}^N} \\
  & \lesssim \ell \int_0^t  e^{-C\tau} \sqrt{N^{-d} \sum_{y_1 \sim y_2} \int_{\sX_N} g(\eta_{y_1}) \left| \mathbb{E} \left( 1_{x_+(t) \not = x_-(t)} 1_{\# \text{jumps} \ge \beta t N^2} \right) \right|^2 \dd \vartheta_{f_{\tau, x}}^N} \dd\tau  
\end{align*}
where we have used the Lipschitz property of $\Phi_{t-\tau}$ in order to extract the $N^{-2d}$ decay inside the square root, since the $\ell^1$ distance between $\eta_{t-\tau}$ and $\eta_{t-\tau}^{x_-(t-\tau) x_+(t-\tau)}$ is always $2$.

Then in dimension $d=1$, the probability of not returning to zero in $\beta (t-\tau)N^2$ jumps is $\mathcal O((t-\tau)^{-1/2} N^{-1})$, see~\cite[Prop~4.2.4, p.80]{Lawler_Limic_2010}. 
This would immediately give $J^{\text{NR}}_t \lesssim \ell N^{-1/2}$ if it was not for the factor $g(\eta_{y_1})$ in the integrand which is unbounded. We control it in the usual way by using exponential moments on the Gibbs measure, splitting $\eta_{y_1} \le M$ and $\eta_{y_1} > M$ and finally optimising $M=M(N) \sim \ln((t-\tau)N^2)$. We leave this easy calculation to the reader, which leads to
\begin{equation}
  \label{eq:1d-JtN}
  J^{\text{NR}}_t \lesssim \frac{\ell \ln N}{\sqrt{N}}.
\end{equation}

In dimension $d=2$, the probability of not returning to zero in $\beta (t-\tau)N^2$ jumps is $\mathcal O((\ln((t-\tau)N^2)^{-1})$, see~\cite[Prop~4.2.4, p.80]{Lawler_Limic_2010}. This would immediately give $J^{NR}_t \lesssim \ell (\ln N)^{-1/2}$ if it was not for the unbounded factor $g(\eta_{y_1})$ in the integrand. We control it again by using exponential moments on the Gibbs measure, splitting $\eta_{y_1} \le M$ and $\eta_{y_1} > M$ and finally optimising $M=M(N) \sim \sqrt{|\ln((t-\tau)N^2)|}$. We again leave this calculation to the reader, which leads to
\begin{equation}
  \label{eq:2d-JtN}
  J^{\text{NR}}_t \lesssim \frac{\ell}{(\ln N)^{1/4}}.
\end{equation}

\subsection{Step 5. Local limit theorem at the local equilibrium}

To control the term $\tilde J_t ^N$
in~\eqref{eq:decomp-consis2} we use the so-called \textit{equivalence of ensembles}. It is proved via the local limit theorem~\ref{theo:llt} applied in the box $\cC^\ell_x$ to the local equilibrium $\vartheta_{f_{\tau,x}} ^\ell$, which is a product identically distributed measure with one-site mass parameter $f_{\tau,x}=f(\tau,x/N)$ bounded from above and away from zero, and to the functions $g(\eta_{y_0})$ which only depends on the value of $\eta$ at some site $y_0 \in \cC_x^\ell$, and then perform the average over $y_0 \in \cC_x^\ell$ to deduce
\begin{equation}
  \label{eq:equiv-ens}
  \left[\Pi_x^\ell \langle g(\eta_\cdot) \rangle_{\cC^\ell_x} \right](m) 
  = {\bf E}_{\vartheta^\ell_m}\left[
    \langle g\left(\eta_\cdot \right) \rangle_{\cC^\ell_x} \right] + \cO\left(
    \ell^{-d} \right) = \sigma(m) + \cO\left(
    \ell^{-d+0} \right)
\end{equation}
when $m := \langle \eta \rangle_{\cC^\ell_x} \in [m_0,m_1]$ for some given $0<m_0<m_1<\infty$. We choose $m_0$ small enough and $m_1$ large enough so that the $\vartheta_{f_{\tau,x}} ^\ell$-probability that $m$ falls outside $[m_0,m_1]$ is exponentially small in $\ell$. Inserting this bound into our formula for $ \hat A^N_{\tau,x}$, we get
\begin{align*}
  & \Pi^\ell_x \hat A^N_{\tau,x}
    =  \left\{ \left[\Pi_x^\ell \langle g(\eta_\cdot) \rangle_{\cC^\ell_x} \right] -\sigma\left( f_\tau\left(\frac{x}{N}\right) \right) -\sigma'\left( f_\tau\left(\frac{x}{N}\right) \right) \left[\langle \eta \rangle_{\cC_x ^\ell} - f_\tau\left(\frac{x}{N}\right) \right] \right\}    \frac{\mathbf{L}_\infty (f)\left(\frac{x}{N}\right)}{\sigma\left(    f_\tau\left(\frac{x}{N}\right) \right)} \\
  & = \left\{ \sigma \left( \langle \eta \rangle_{\cC_x^\ell} \right) -    \sigma\left( f_\tau\left(\frac{x}{N}\right) \right) - \sigma'\left( f_\tau\left(\frac{x}{N}\right) \right) \left[ \langle \eta \rangle_{\cC_x ^\ell} - f_\tau\left(\frac{x}{N}\right) \right] \right\} \frac{\mathbf{L}_\infty (f)\left(\frac{x}{N}\right)}{\sigma\left( f_\tau\left(\frac{x}{N}\right) \right)} + \cO\left(  \ell^{-d+0} \right) \\
  &  = \cO \left( \left| \langle \eta^{(\ell)}\rangle_{\cC_x^\ell} - f_\tau \left(\frac{x}{N} \right) \right|^2 e^{-C \tau} \right) + \cO(\ell^{-d+0})
\end{align*}
(abusing slightly notation by incorporating the exponential small set $m \not \in [m_0,m_1]$ into the error). The Lipschitz regularity of $\Phi_{t-s}$ implies $\Pi_x ^\ell {\bf P}_{\cC_x ^{\ell}} \tilde \Phi_{t,\tau} = \cO(\ell^d N^{-d})$, since we recall that $\tilde \Phi_{t,\tau} := e^{(t-\tau)\mathcal{L}_N } \Phi-\mathbb{E}_{\vartheta_{f_\tau}^N} ( e^{(t-\tau)\mathcal{L}_N }\Phi)$. 
We however need a better control.

We claim that this Lipschitz regularity of $\Phi_{t-s}$ also implies a Lipschitz regularity of its averaged projection $\Pi_x ^\ell {\bf P}_{\cC_x ^{\ell}} \Phi_{t-s}$ with constant $\ell^d N^{-d}$, with respect to the local mass $\langle \eta^{(\ell)} \rangle_{\cC_x^\ell}$. Let us prove this fact. We reason in the $\ell$-box. Given $0 \le m \le m' <+\infty$, pick any pair of configuration $(\eta_0,\zeta_0)$ with $\langle \eta_0 \rangle_{\cC_x} =m$, $\langle \zeta_0 \rangle_{\cC_x} =m'$ and $\eta_0 \le \zeta_0$ (such configuration trivially exists since $m \le m'$). We then consider the initial coupling $\delta_{(\eta_0,\zeta_0)}$ on $\Omega_m \times \Omega_{m'}$ which has $\ell^1$ cost $m'-m$, and we evolve it along the flow of the coupling operator $e^{t \tilde \cL_N}$. The marginals respectively converge to $\vartheta^{\ell,m}$ and $\vartheta^{\ell,m'}$ (convergence to equilibrium of the original evolution). Since the evolution by the coupling operator satisfies our microscopic stability estimate, we have $W_1(\vartheta^{\ell,m},\vartheta^{\ell,m'})\le C_S(m'-m)$. An optimal coupling $\pi$ thus satisfies
\begin{equation*}
  \int_{\Omega_m \times \Omega_{m'}} \left( \ell^{-d} \sum_{z \in \cC^\ell_x} |\eta_z - \zeta_z| \right) \dd \pi(\eta,\zeta) \le C_S(m'-m)
\end{equation*}
and we compute
\begin{align*}
  & \left| 
  \Pi_x ^\ell {\bf P}_{\cC_x ^{\ell}} \Phi_{t-\tau}(m') - \Pi_x ^\ell {\bf P}_{\cC_x ^{\ell}}  \Phi_{t-\tau} (m) \right| \\
  & = \left| \int_{\Omega_{m'}} {\bf P}_{\cC_x ^{\ell}}\Phi_{t-\tau} (\zeta) \dd \vartheta^{\ell,m'}(\zeta)- \int_{\Omega_m}  {\bf P}_{\cC_x ^{\ell}} \Phi_{t-\tau} (\eta) \dd
    \vartheta^{\ell,m}(\eta) \right| \\
  & \le \int_{\Omega_m \times \Omega_{m'}} \left|
    {\bf P}_{\cC_x ^{\ell}} \Phi_{t-\tau} (\zeta) -
  {\bf P}_{\cC_x ^{\ell}}  \Phi_{t-\tau} (\eta) \right|
    \dd \pi(\eta,\zeta) \\
  & \le \frac{\ell^d}{N^d} \int_{\Omega_m \times \Omega_{m'}}
   \left( \ell^{-d} \sum_{z \in \cC^\ell_x} |\eta_z - \zeta_z| 
   \right) 
    \dd \pi(\eta,\zeta) \le C_S \frac{\ell^d}{N^d} |m'-m|.
\end{align*}

Using then the almost zero-average
property~\eqref{eq:averagehatA} of $\hat A^N _{\tau,x}$, we write
\begin{align}
  \nonumber
  & \tilde J_t ^N = (2\ell+1)^d \sum_{x \in \cR_N^{\ell} } \int_0 ^t \int_{\sX_\ell} \left( \Pi_x ^\ell {\bf P}_{\cC_x ^{\ell}}\tilde \Phi_{t,\tau} \right) \left( \Pi_x ^\ell \hat A^N_{\tau,x} \right) \dd   \vartheta_{f_{\tau,x}}^\ell \\
  \nonumber
  & \lesssim \ell^d \sum_{x \in \cR_N^{\ell}}  \int_0 ^t \int_{\sX_\ell} \left| \left( \Pi_x ^\ell {\bf P}_{\cC_x ^{\ell}}  \tilde\Phi_{t,\tau}  \right)\left(\langle \eta^{(\ell)} \rangle_{\cC_x^\ell}\right) - \left( \Pi_x ^\ell  {\bf P}_{\cC_x ^{\ell}} \tilde \Phi_{t,\tau} \right) \left(  f_\tau\left(\frac{x}{N}\right) \right) \right| \left( \Pi^\ell_x \hat A^N_{\tau,x} \right) \dd \vartheta_{f_{\tau,x}} ^\ell  +  \frac{\ell^{d+1}}{N}
  \\ \label{eq:estimJtilde}
  & \lesssim \frac{\ell^{2d+0}}{N^d} \sum_{x \in \cR_N^{\ell}} \int_{\sX_{\ell}} \left| \langle \eta^{(\ell)} \rangle_{\cC_x^\ell} - f_\tau \left(\frac{x}{N} \right) \right|^3 e^{-C\tau}  \dd \vartheta_{f_{\tau,x}} ^\ell + \frac{\ell^{d+1}}{N}   \lesssim \ell^{-\frac{d}{2}+0} + \frac{\ell^{d+1}}{N}, 
\end{align}
where for the last term we employed the local law of large numbers in the $\ell$ box, to extract an $\ell^{-3d/2}$ decay from the integrand. We note that Step 5 holds for any dimension $d\geq 1$.

\subsection{Step 6. Synthesis of the errors}

In dimension $d=1$, by combining~\eqref{eq:decomp-consis2}-\eqref{eq:1d-JtN}-\eqref{eq:estimJtilde}, we get
\begin{align*}
  I_t ^N
  & \lesssim \cO \left( \frac{\ell^{\frac{3d}{2}+1}}{N} \ln \left( \frac{N }{\ell^{d+1}} \right) \right) +\cO \left( \frac{\ell}{\sqrt{N}} \ln N \right) +\cO\left(  \ell^{-\frac{d}{2}+0} \right) \\
  & \lesssim \cO \left( \frac{ \ell}{\sqrt{N} } \ln N \right) +\cO\left( \ell^{-\frac{d}{2}+0} \right)  
\end{align*}
and after optimising over $\ell$, choosing $\ell = N^{1/3}$, the final consistency rate is $I_t^N  = \mathcal{O} (N^{-1/6+0} \ln(N) )$. In dimension $d=2$, by combining~\eqref{eq:decomp-consis2}-\eqref{eq:2d-JtN}-\eqref{eq:estimJtilde}, we get 
\begin{align*}
  I_t ^N
 & \lesssim  \cO \left( \frac{\ell^{\frac{3d}{2}+1}}{N} \ln \left( \frac{N }{\ell^{d+1}} \right) \right) +\cO \left( \frac{\ell}{ [ \ln(N) ]^{1/4}}  \right) +\cO\left( \ell^{-\frac{d}{2}+0} \right) \\ 
  & \lesssim \cO \left(  \frac{\ell}{ [ \ln(N) ]^{1/4}} \right) +\cO\left( \ell^{-\frac{d}{2}+0} \right). 
\end{align*}
We optimise by choosing $\ell = [ \ln(N) ]^{1/8}$ so that $I_t^N= \mathcal{O}([ \ln(N) ]^{-1/8+0})$.

\begin{remark}  
  In the case of dimension $1$, when $\Phi_0$ has the additional structure $\Phi_0(\eta)=N^{-1} \sum_{x \in \T_N} \eta_x \phi_x$ with $\phi \in C^1(\T)$, (1) by using a more general Poincaré inequality valid on the \textbf{non-flattened measure} one gets rid altogether of the error~\eqref{eq:decomp-consis2}, (2)  by using the ``method $1$'' one gets the better rate $\cO(\ell \ln N/N)$ in~\eqref{eq:1d-JtN}, (3) by using higher-order versions of the local limit theorem~\ref{theo:llt} one expect to get the better rate $\cO(\ell^{-d})$ in~\eqref{eq:estimJtilde}, modulo a term that is killed by orthogonality, which would give (asymptotically) the optimal rate $\cO(N^{-1/2+0})$. Such improvement will be considered in further work. Note that we do not know if a jump distance monotonicity (as in ``method $1$'') holds in dimension $d \ge 2$.
\end{remark}
 
\section{Entropic consistency}
\label{sec:entropy}

The $L^1_t$ estimate of $\mathsf R_1$ for the ZRP was established in the previous section. We now prove the $L^\infty_t$ estimate over $\mathsf R_2$, and then~\textbf{(H4)}, for the ZRP.

\subsection{Estimate on $\mathsf R_2$ for the ZRP}
We consider the two functions in~\eqref{eq:lip-test} that are to be estimated in the Lipschitz seminorm along time. We start with  $N^{-d} \ln \left( G_t^N(\eta) \right) $. Since the density $G^N_t$ of the local Gibbs measure is 
\begin{align*}
  G_t^N(\eta) = \frac{\dd \vartheta_{f_t}^N }{\dd \vartheta_{f_\infty}^N}(\eta) = \prod_{x\in\T^d_N} \frac{Z(\sigma(f_\infty))}{Z(\sigma(f_{t,x}))}  \bigg(\frac{\sigma(f_{t,x})}{\sigma(f_\infty)}\bigg)^{\eta_x}
\end{align*} 
we deduce
\begin{equation*}
  \frac{\ln \left( G_t^N (\eta)\right)}{N^d} = \frac{1}{N^d} \sum_{x\in\T^d_N} \left[ \ln \left(  \frac{ Z(\sigma(f_\infty)) }{ Z(\sigma(f_{t.x}))}  \right)   + \eta_x \ln \left( \frac{\sigma(f_{t,x})}{\sigma(f_\infty)}  \right) \right]
\end{equation*}
which satisfies the Lipschitz condition uniformly in time, by using the uniform bounds on $\sigma(f_{t.x})$ and $Z$. We then consider $N^{-d} (\mathcal{L}_N \sqrt{G_t^N})/\sqrt{G_t^N}$. Explicit calculations yield
\begin{align*}
\sqrt{G_t^N} = \prod_{x\in\T^d_N} \sqrt{\left(\frac{\sigma\left( f_{t,x}\right)}{\sigma(f_\infty)} \right)^{\eta_x} \frac{Z(\sigma(f_\infty))}{Z(\sigma(f_t(x/N)))}}
\end{align*} 
so that 
\begin{align*}
  \mathcal{L}_N \sqrt{G_t^N} (\eta) = N^2 \sum_{x \in \T^d_N} \sum_{|e|=1}    g(\eta_x) \bigg[ \sqrt{ \frac{ \sigma\big(f_t\big(\frac{x+e}{N}\big)\big) }{       \sigma\big(f_t\big(\frac{x}{N}\big)\big) } } - 1 \bigg] \sqrt{G_t^N}  (\eta)
\end{align*}
and hence 
\begin{equation*}
  \frac{\mathcal{L}_N \sqrt{G_t^N}(\eta) }{N^d \sqrt{G_t^N}(\eta)}
  = \frac{1}{N^d} \sum_{x\in\T^d_N}     \frac{g(\eta_x)}{\sqrt{\sigma\left(f_{t,x}\right)}}
    N^2 \sum_{y \sim x} \left[ \sqrt{\sigma\left(f_{t,y}\right)} - \sqrt{\sigma\left(f_{t,x}\right)} \right] = \frac{1}{N^d}  \sum_{x\in\T^d_N}  g(\eta_x) \phi_x 
\end{equation*}
with the following test function at each site
\begin{equation*}
  \phi_x := \frac{1}{\sqrt{\sigma\left(f_{t,x}\right)}} N^2 \sum_{y \sim x} \left[ \sqrt{\sigma\left(f_{t,y}\right)} - \sqrt{\sigma\left(f_{t,x}\right)} \right]
\end{equation*}
which is bounded in $C^1$ thanks to the regularity of $\sigma$ and $f$ (note that the last sum is a discrete Laplacian). We have therefore proved that $\mathsf R_2(t)$ is uniformly bounded in time.

\subsection{Proof of \textbf{(H4)} for the ZRP}

We now prove~\eqref{eq:consist-ent} in~\textbf{(H4)}. Let us compute
\begin{align*}
  H^N\left( \vartheta_{f_t}^N| \vartheta^N _{f_\infty} \right)
  & = \frac{1}{N^d} \sum_{x \in \T^d_N} \int_{\sX_N} \left[ \eta_x \ln \left( \frac{\sigma\left( f_{t,x}\right)}{\sigma(f_\infty)} \right) + \ln \left( \frac{Z(\sigma(f_\infty))}{Z\left(\sigma\left(f_{t,x} \right)\right)} \right) \right] \dd \vartheta_{f_t} ^N(\eta)
  \\
  & = \frac{1}{N^d} \sum_{x \in \T^d_N} \left[f_{t,x}  \ln \left( \frac{\sigma\left( f_{t,x}\right)}{\sigma(f_\infty)}
    \right) + \ln \left( \frac{Z(\sigma(f_\infty))}{Z\left(\sigma\left( f_{t,x}\right)\right)} \right) \right] \\
  & = \int_{\T^d} \left( f_t \ln \frac{\sigma(f_t)}{\sigma(f_\infty)}
    + \ln \frac{Z(\sigma(f_\infty))}{Z(\sigma(f_t))} \right) \dd x
    + \cO\left(\frac{1}{N}\right) 
\end{align*}
by using the periodicity and $C^{3}$ regularity of $f$,
and
\begin{align*}
  & \frac{4}{N^d} \int_0 ^t \int_{\sX_N}
  \sqrt{G^N_\tau}  \cL_N \left( \sqrt{G^N_\tau} \right) \dd
  \vartheta^N _{f_\infty} \dd \tau \\
  & = \frac{4 N^2}{N^d} \int_0 ^t \int_{\sX_N}\sum_{x,y \in \T^d_N} p(y-x) g(\eta_x) \left[\sqrt{\frac{\sigma\left( f_{t,y}\right)}{\sigma\left( f_{t,x}\right)}} - 1 \right] \dd \vartheta^N _{f_t} \dd \tau \\
  & = \frac{4 N^2}{N^d} \sum_{x,y \in \T^d_N} \int_0 ^t  p(y-x) \left[ \sqrt{\sigma\left( f_{t,y}\right)} - \sqrt{\sigma\left(   f_{t,x}\right)} \right] \sqrt{\sigma\left(   f_{t,x}\right)} \dd \tau
\end{align*}
which can be approximated by higher-order Taylor expansion and employing the decay of $f_t$. 
We then deduce 
\begin{align*}
  & \frac{4}{N^d} \int_0 ^t \int_{\sX_N}
    \sqrt{G^N_\tau}  \cL_N \left( \sqrt{G^N_\tau} \right) \dd
    \vartheta^N _{f_\infty} \dd \tau \\
  & =  \frac{4}{N^{d}} \sum_{x \in \T^d_N} \int_0 ^t
    \underbrace{\left( \sum_{y \in
    \T^d_N} p(y-x) (y-x) \right)}_{=0} \cdot \left(
    \sqrt{\sigma(f_t)} \nabla \sqrt{\sigma(f_t)} \right)
    \left( \frac{x}{N}\right)  \dd \tau \\
  & \quad + \frac{4}{N^{d}}
    \sum_{x \in \T^d_N} \int_0 ^t
    \underbrace{\left( \sum_{y \in
    \T^d_N} \frac12 p(y-x) (y-x)^{\otimes 2} \right)}_{=A} : \left(
    \sqrt{\sigma(f_t)} \nabla^2 \sqrt{\sigma(f_t)} \right)
    \left( \frac{x}{N}\right)  \dd \tau + \cO
    \left(\frac{1}{N} \right) \\
  & = - \int_0 ^t \int_{\T^d} \frac{A : \nabla \sigma(f_t)
    \otimes \nabla \sigma(f_t)}{\sigma(f_t)} \dd x \dd \tau + \cO
    \left(\frac{1}{N} \right)
\end{align*}
which proves ~\eqref{eq:consist-ent} with $\epsilon_{NL}(N) =
\cO(N^{-1})$. Note also that 
\begin{align*}
  \dt \int_{\T^d} \left( f_t \ln \frac{\sigma(f_t)}{\sigma(f_\infty)}
  + \ln \frac{Z(\sigma(f_\infty))}{Z(\sigma(f_t))} \right) \dd x
  = - \int_{\T^d} \frac{A : \nabla \sigma(f_t)
    \otimes \nabla \sigma(f_t)}{\sigma(f_t)} \dd x
\end{align*}
where we used the identity $Z'(\sigma(f_t))/Z(\sigma(f_t)) = f_t/\sigma(f_t)$.
\appendix

\section{Quantitative local limit estimates}

The content of this appendix closely follows~\cite{MR388499, KL99}. It is included for the reader's convenience and the sake of completeness. We provide a quantitative proof of the equivalence of ensembles for the ZRP (the argument applies to more general classes of invariant measures). The proof is based on the local limit theorem in~\cite{MR388499} and~\cite[Appendix 2]{KL99}. We consider our $d$-dimensional lattice $\cC^\ell_x$ with length $\ell$ and our product measure $\vartheta^\ell_\rho$ with one-site mass $\rho$, i.e. $\vartheta^\ell_\rho[\eta_{y_0}] = \rho$ for $y_0 \in \cC_x^\ell$. 
We denote by $\vartheta^{\ell,m}$ this invariant measure conditioned to the hyperplan $\langle \eta \rangle_{\cC_x ^\ell} = m$. Importantly one can check (by cancelling terms in the following fraction) that $\vartheta^{\ell,m}$ is independent of $\rho>0$:
\begin{equation*}
  \int_{\cC^\ell_x} \varphi(\eta^{(\ell)}) \dd \vartheta^{\ell,m}(\eta^{(\ell)})
   = \frac{\ds \int_{\langle \eta^{(\ell)} \rangle_{\cC^\ell_x} = m} \varphi(\eta^{(\ell)}) \dd \vartheta^\ell_\rho(\eta^{(\ell)})}{\ds \int_{\langle \eta^{(\ell)} \rangle_{\cC^\ell_x} = m} \dd \vartheta^\ell_\rho(\eta^{(\ell)})} 
   = \frac{\ds \sum_{\langle \eta^{(\ell)} \rangle_{\cC^\ell_x} = m} \varphi(\eta^{(\ell)}) \prod_{y \in \cC_x^\ell} \frac{1}{g(\eta_y)!}}{\ds \sum_{\langle \eta^{(\ell)} \rangle_{\cC^\ell_x} = m} \prod_{y \in \cC_x^\ell} \frac{1}{g(\eta_y)!}}.
\end{equation*}

\begin{theorem}
  \label{theo:llt}
  Given $0<m_0<m_1<\infty$ and $g(\eta_{y_0})$ a function depending only on $\eta$ at one site $y_0 \in \cC_x ^\ell$, there exists a constant $C_E=C_E(m_0,m_1)$ such that for any $m \in [m_0,m_1]$
  \begin{equation*}
    \left|  \mathbb{E}_{\vartheta^{\ell,m}}[g(\eta_{y_0})] - \mathbb{E}_{\vartheta^\ell_m}[g(\eta_{y_0})] \right| \leq C_E \ell^{-d+0}.
  \end{equation*}
\end{theorem}

\begin{proof}[Proof of Theorem~\ref{theo:llt}]
  The proof is based on establishing a Taylor expansion of $m' \mapsto \vartheta^\ell_m(\{\langle \eta^{(\ell)} \rangle_{\cC^\ell_x}=m'\})$ around $m$, and then using it on $\cC_x ^\ell$ and $\cC^\ell_x$ deprived from one point in order to estimate the conditional measure applied to a one-site function. The Taylor expansion reads, for $m'=m-k \ell^{-d}$, $k \in \{0,1,\dots,\ell^d\}$ (the fluctuations are of the order of the number of particles at one site):
  \begin{equation}
    \label{eq:taylor-ensembles}
    \vartheta^\ell_m\left(\left\{\langle \eta^{(\ell)} \rangle_{\cC^\ell_x}=m - k \ell^{-d} \right\}\right) = \frac{e^{-\frac{k^2}{c_2(m) \ell^d}}}{\sqrt{4 \pi c_2(m) \ell^d}} \left[ 1 - \cO\left( \ell^{-d+0} \right) \right]
  \end{equation}
  with $c_2(m) := \mathbb E_{n_m}(\eta_0^2)/2$, uniformly for $m \le m_0$ bounded above. Before proving~\eqref{eq:taylor-ensembles}, let us show how it implies the result. By Fubini we have 
  \begin{equation*}
     \mathbb{E}_{\vartheta^{\ell,m}}[g(\eta_{y_0})] = \int_\N g(\eta_{y_0}) \frac{\tilde \vartheta^\ell_m\left(\left\{\langle \tilde \eta^{(\ell)} \rangle_{\cC^\ell_x}=m-\eta_{y_0} \ell^{-d} \right\}\right)}{\vartheta^\ell_m\left(\left\{\langle \eta^{(\ell)} \rangle_{\cC^\ell_x}=m \right\}\right)} \dd n_m(\eta_{y_0})  
   \end{equation*}
   where $\tilde \vartheta^\ell_m$ is the product measure on $\cC^\ell_x \setminus \{y_0\}$ (which has $\ell^d-1$ sites) with local mass parameter $m$ (we have chosen $\rho=m$ when writing the conditional measure). Hence
   \begin{equation*}
     \mathbb{E}_{\vartheta^{\ell,m}}[g(\eta_{y_0})] - \mathbb{E}_{\vartheta^\ell_m}[g(\eta_{y_0})] = \int_\N g(\eta_{y_0}) \left[ \frac{\tilde \vartheta^\ell_m\left(\left\{\langle \tilde \eta^{(\ell)} \rangle_{\cC^\ell_x}=m-\eta_{y_0} \ell^{-d} \right\}\right)}{\vartheta^\ell_m\left(\left\{\langle \eta^{(\ell)} \rangle_{\cC^\ell_x}=m \right\}\right)} -1 \right] \dd n_m(\eta_{y_0}).
   \end{equation*}
   Since the above expansion~\eqref{eq:taylor-ensembles} can be applied in $\cC^\ell_x \setminus \{ y_0 \}$ (the asymptotic is the same since the number of sites is $\ell^d-1 \sim \ell^d$) for the numerator, as well as to the denominator:
   \begin{equation*}
     \frac{\tilde \vartheta^\ell_m\left(\left\{\langle \tilde \eta^{(\ell)} \rangle_{\cC^\ell_x}=m-\eta_{y_0} \ell^{-d} \right\}\right)}{\vartheta^\ell_m\left(\left\{\langle \eta^{(\ell)} \rangle_{\cC^\ell_x}=m \right\}\right)} = e^{-\frac{\eta_{y_0}^2}{c_2(m) \ell^{d/2}}} \frac{\sqrt{\ell^d-1}}{\ell^d} \left[ 1 + \cO \left( \eta_{y_0}^2 \ell^{-d+0} \right) \right] = 1 + \cO\left( \ell^{-d+0} \right)
   \end{equation*}
   thanks to the exponential moments of $n_m(\eta_{y_0})$, which implies the desired estimate.

   We now prove the Taylor expansion. The law $\mu_S$ of the integer-valued random variable $S:=\sum_{y \in \cC^\ell_x} \eta_y$ has $2\pi$-periodic characteristic function $\hat \mu_S(t) = (\hat n_m(t))^{\ell^d}$, where $\hat n_m(t)$ is the $2\pi$-periodic characteristic function of $n_m$, given $t \in \R$. Given $m'=m-k\ell^{-d}$, we apply the periodic inverse Fourier formula:
   \begin{align*}
     \vartheta^\ell_m\left(\left\{\langle \eta^{(\ell)} \rangle_{\cC^\ell_x} =m' \right\}\right)
     & = \mu_S(m'\ell^d) = \frac{1}{2 \pi} \int_{-\pi} ^\pi e^{-it m' \ell^d} \dd \hat \mu_S(t) \\
     & = \frac{1}{2 \pi} \int_{-\pi} ^\pi e^{-it m' \ell^d} \left[\hat n_m(t)\right]^{\ell^d} \dd t \\
     & = \frac{1}{2 \pi \ell^{d/2}} \int_{-\pi \ell^{d/2}} ^{\pi \ell^{d/2}} e^{-it m'\ell^{\frac{d}2}} \left[\hat n_m\left(t\ell^{-\frac{d}2}\right)\right]^{\ell^d} \dd t \\
     & = \frac{1}{2 \pi \ell^{d/2}} \int_{-\pi \ell^{d/2}} ^{\pi \ell^{d/2}} e^{-it m' \ell^{\frac{d}2}} 
       \mathbb E_{n_m}\left[ e^{it \eta_0 \ell^{-\frac{d}2}} \right]^{\ell^d} \dd t. 
   \end{align*}
   and, given $\var>0$ and $\delta \in (0,\pi)$ both small, to be chosen later, we decompose the last $t$-integral into three parts:
   \begin{multline*}
     \vartheta^\ell_m\left(\left\{\langle \eta^{(\ell)} \rangle_{\cC^\ell_x} =m' \right\}\right) = \frac{1}{2 \pi \ell^{d/2}} \int_{|t| \le \ell^\var} \left( \dots \right) \dd t + \frac{1}{2 \pi \ell^{d/2}} \int_{\ell^\var <|t| \le \delta \ell^{d/2}} \left( \dots \right) \dd t \\ + \frac{1}{2 \pi \ell^{d/2}} \int_{\delta \ell^{d/2} <|t| \le \pi \ell^{d/2}} \left( \dots \right) \dd t =: I_1(\ell) + I_2(\ell) + I_3(\ell).
   \end{multline*}
   We now estimate each part successively. To estimate the first part, we observe that on the domain $\{|t| \le \ell^\var \}$ with $\var < d/2$, we have, by expanding the exponential and using the exponential moments of $n_m$,
   \begin{align}
     \nonumber
     \mathbb E_{n_m}\left[ e^{it \eta_0 \ell^{-\frac{d}2}} \right]^{\ell^d}
     & = \left[ 1 + it m \ell^{-\frac{d}2} - c_2 t^2 \ell^{-d} - i c_3 t^3 \ell^{-\frac{3d}{2}} + \cO\left(t^4 \ell^{-2d}\right) \right]^{\ell^d} \\
     \label{eq:exp-ensembles}
     & = \exp \left[ it m \ell^{\frac{d}2} - c_2 t^2 - i c_3 t^3 \ell^{-\frac{d}{2}} + \cO\left(t^4 \ell^{-d}\right) \right] 
   \end{align}
   with $c_2$ defined above and $c_3 := \mathbb E_{n_m}(\eta_0^3)/6$, so that
   \begin{align*}
     I_1(\ell)
     & = \frac{1}{2 \pi \ell^{d/2}} \int_{|t| \le \ell^\var} \exp \left[ it (m-m') \ell^{\frac{d}2} - c_2 t^2 - i c_3 t^3 \ell^{-\frac{d}{2}} + \cO\left(t^4 \ell^{-d}\right) \right] \dd t \\
     & = \frac{1}{2 \pi \ell^{d/2}} \int_{|t| \le \ell^\var} \exp \left[ it k \ell^{-\frac{d}2} - c_2 t^2 - i c_3 t^3 \ell^{-\frac{d}{2}} + \cO\left(t^4 \ell^{-d}\right) \right] \dd t \\
     & = \frac{e^{-\frac{k^2}{c_2\ell^d}}}{2 \pi \ell^{d/2}} \int_{|t| \le \ell^\var} \exp \left[ - c_2 \left(t - ik \ell^{-d/2} c_2^{-1} \right)^2 - i c_3 t^3 \ell^{-\frac{d}{2}} + \cO\left(t^4 \ell^{-d}\right) \right] \dd t
   \end{align*}
   and finally using $\exp[\cO(t^4 \ell^{-d})] = 1 + \cO( \ell^{4\var -d})$ when $|t| \le \ell^\var$ and, by complex contour deformation and estimating the vertical parts of the contour
   \begin{multline*}
     \int_{|t| \le \ell^\var} \exp \left[ - c_2 \left(t - ik \ell^{-d/2} c_2^{-1} \right)^2 - i c_3 t^3 \ell^{-\frac{d}{2}} \right] \dd t \\
     = \int_\R \exp \left[ - c_2 t^2 - i c_3 \left( t + i k \ell^{-d/2} c_2^{-1} \right)^3 \ell^{-\frac{d}{2}} \right] \dd t + \cO \left( e^{-(c_2/2) \ell^\var} \right)
   \end{multline*}
   and since $\Re \exp[ - i c_3 \left( t + i k \ell^{-d/2} c_2^{-1} \right)^3 \ell^{-\frac{d}{2}}] = 1 + \cO(\ell^{2\var-d})$, we deduce by taking the real part (the loss on the power of $\ell$ could easily be reduced to a logarithm)
   \begin{equation*}
     I_1(\ell) = \frac{e^{-\frac{k^2}{c_2\ell^d}}}{2 \pi \ell^{d/2}} \left[ \int_\R e^{-c_2 t^2} \dd t + \cO( \ell^{3\var -d}) \right] = \frac{e^{-\frac{k^2}{c_2\ell^d}}}{\sqrt{4 \pi c_2 \ell^d}} \left[ 1 + \cO( \ell^{3\var -d}) \right].
   \end{equation*}
   To control $I_2(\ell)$ we use again the expansion~\eqref{eq:exp-ensembles} and take advantage of the fact that $t$ is away from zero but also less than $\delta \ell^{d/2}$ which implies that $t^4 \ell^d \le \delta^2 t^2$ to get
   \begin{align*}
     \left| I_2(\ell) \right| 
     & = \frac{1}{2 \pi \ell^{d/2}} \left| \int_{\ell^\var \le |t| \le \delta \ell^{d/2}} \exp \left[ it k \ell^{-\frac{d}2} - c_2 t^2 - i c_3 t^3 \ell^{-\frac{d}{2}} + \cO\left(t^4 \ell^{-d}\right) \right] \dd t \right| \\
     & \le \frac{1}{2 \pi \ell^{d/2}} \int_{\ell^\var \le |t| \le \delta \ell^{d/2}} \exp \left[ - c_2 t^2 + \cO\left(\delta t^2 \right) \right] \dd t = \cO \left( e^{-(c_2/2) \ell^\var} \right)
   \end{align*}
   for $\delta>0$ small enough. Finally to control $I_3(\ell)$ we will use the concentration of measure of $\hat n_m(t)$. We write (undoing the previous change of variable in $t$)
   \begin{equation}
     \label{eq:control-I3}
     \left| I_3(\ell) \right| 
     = \frac{1}{2 \pi} \left| \int_{\delta <|t| \le \pi} e^{-it m' \ell^d} \hat \mu_S(t) \dd t \right| \le  \frac{1}{2 \pi} \int_{\delta <|t| \le \pi} \left| \hat n_m(t) \right|^{\ell^d} \dd t.
   \end{equation}
   Then we calculate (denoting $p_m ^j = \sigma(m)^j/(g(j)!Z(\sigma(m))$ the coefficients of the measure $n_m$ on $\N$)
   \begin{align*}
     \left| \hat n_m(t) \right|^2 = \left| \sum_{j \ge 0} e^{itj} p_m ^j \right|^2 = \sum_{j_1,j_2 \ge 0} e^{it(j_1-j_2)} p_m ^{j_1} p_m^{j_2} = \sum_{j \ge 0} \left( \frac{e^{it j} + e^{-itj}}{2} \right) \tilde p_m ^j = \sum_{j \ge 0} \cos(tj) \tilde p_m ^j 
   \end{align*}
   with the notation $\tilde p_m ^j := \sum_{j_1,j_2\ge 0} p_m ^{j_1} p_m ^{j_2} \ge 0$ (we have symmetrised the sums). Since we have $\sum_{j \ge 0} \tilde p_m ^j = (\sum_{j \ge 1} p_m ^j)^2=1$, we deduce
   \begin{align*}
     \left| \hat n_m(t) \right|^2 -1 = \sum_{j \ge 0} \left[ \cos(tj) -1 \right] \tilde p_m ^j \le - \beta <0
   \end{align*}
   when $|t| \in [\delta,\pi]$, since $p_m^1$ bounded below away from zero uniformly in $m \in [m_0,m_1]$, for some $\beta \in (0,1)$ depending on $\delta$, $m_0$ and $m_1$. Therefore~\eqref{eq:control-I3} implies
   \begin{equation*}
     \left| I_3(\ell) \right| 
     \le  \frac{1}{2 \pi} \int_{\delta <|t| \le \pi} \left| \hat n_m(t) \right|^{\ell^d} \dd t \le \frac{1}{2} (1-\beta)^{\frac{\ell^d}{2}}
   \end{equation*}
   and the $I_3$ term decays exponentially fast in $\ell$. 
\end{proof}

\bibliographystyle{alpha}
\bibliography{bibliography}

\begin{thebibliography}{GOVW09}

\bibitem[CY92]{YauH-1}
C.-C. Chang and H.-T. Yau.
\newblock Fluctuations of one dimensional ginzburg-landau models in
  nonequilibrium.
\newblock {\em Communications in Mathematical Physics}, 145(2):209--234, 1992.

\bibitem[DMOW]{DMOW18a}
D.~Dizdar, G.~Menz, F.~Otto, and T.~Wu.
\newblock The quantitative hydrodynamic limit of the {K}awasaki dynamics.
\newblock Preprint. arXiv:1807.09850.

\bibitem[FGW24]{FGW24}
T.~Funaki, C.~Gu, and H.~Wang.
\newblock Quantitative homogenization and hydrodynamic limit of non-gradient
  exclusion process.
\newblock arXiv:2404.12234, 2024.

\bibitem[GGM22]{Giunti:2022aa}
A.~Giunti, C.~Gu, and J.-C. Mourrat.
\newblock Quantitative homogenization of interacting particle systems.
\newblock {\em The Annals of Probability}, 50(5):1885--1946, 9 2022.

\bibitem[GLT09]{PatriciaLandim09}
P.~Gon{\c{c}}alves, C.~Landim, and C.~Toninelli.
\newblock {Hydrodynamic limit for a particle system with degenerate rates}.
\newblock {\em Annales de l'Institut Henri Poincaré, Probabilités et
  Statistiques}, 45(4):887 -- 909, 2009.

\bibitem[GOVW09]{GOVW09}
N.~Grunewald, F.~Otto, C.~Villani, and M.~Westdickenberg.
\newblock A two-scale approach to logarithmic {S}obolev inequalities and the
  hydrodynamic limit.
\newblock {\em Ann. Inst. Henri Poincar\'{e} Probab. Stat.}, 45(2):302--351,
  2009.

\bibitem[GPV88]{GPV88}
M.~Guo, G.~Papanicolaou, and S.~Varadhan.
\newblock Nonlinear diffusion limit for a system with nearest neighbor
  interactions.
\newblock {\em Comm. Math. Phys.}, 118(1):31--59, 1988.

\bibitem[GV24]{BenVitalii24}
B.~Gess and Konarovskyi V.
\newblock A quantitative central limit theorem for the simple symmetric
  exclusion process.
\newblock arXiv:2408.01238, 2024.

\bibitem[JM18]{JM18}
J.~Jara and O.~Menezes.
\newblock Non-equilibrium fluctuations of interacting particle systems.
\newblock arXiv:1810.09526, 2018.

\bibitem[KL99]{KL99}
C.~Kipnis and C.~Landim.
\newblock {\em Scaling limits of interacting particle systems}, volume 320 of
  {\em Grundlehren der Mathematischen Wissenschaften [Fundamental Principles of
  Mathematical Sciences]}.
\newblock Springer-Verlag, Berlin, 1999.

\bibitem[Lig85]{Liggett1985}
T.~Liggett.
\newblock {\em Interacting Particle Systems}.
\newblock Springer Berlin Heidelberg, 1985.

\bibitem[LL10]{Lawler_Limic_2010}
Gregory~F. Lawler and Vlada Limic.
\newblock {\em Random Walk: A Modern Introduction}.
\newblock Cambridge Studies in Advanced Mathematics. Cambridge University
  Press, 2010.

\bibitem[LSV96]{MR1415232}
C.~Landim, S.~Sethuraman, and S.~Varadhan.
\newblock Spectral gap for zero-range dynamics.
\newblock {\em Ann. Probab.}, 24(4):1871--1902, 1996.

\bibitem[MMM24]{MMM:sequel}
D.~Marahrens, A.~Menegaki, and C.~Mouhot.
\newblock Quantitative estimates in {Y}aus's relative entropy method and
  semigroup viewpoint.
\newblock {\em In progress}, 2024.

\bibitem[Pet75]{MR388499}
V.~Petrov.
\newblock {\em Sums of independent random variables}.
\newblock Ergebnisse der Mathematik und ihrer Grenzgebiete [Results in
  Mathematics and Related Areas], Band 82. Springer-Verlag, New
  York-Heidelberg, 1975.
\newblock Translated from the Russian by A. A. Brown.

\bibitem[Rez91]{Reza}
F.~Rezakhanlou.
\newblock Hydrodynamic limit for attractive particle systems on {${\bf Z}^d$}.
\newblock {\em Comm. Math. Phys.}, 140(3):417--448, 1991.

\bibitem[Yau91]{Yau91}
H.-T. Yau.
\newblock Relative entropy and hydrodynamics of {G}inzburg-{L}andau models.
\newblock {\em Lett. Math. Phys.}, 22(1):63--80, 1991.

\end{thebibliography}

\end{document}